%% file: main_mf.tex
\documentclass[11pt]{article}
\usepackage[left=1in,top=1in,right=1in,bottom=1in,letterpaper]{geometry}

\usepackage[colorlinks, linkcolor=red,filecolor=blue,citecolor=blue,urlcolor=blue]{hyperref}

\usepackage{mathtools} 

\usepackage{url}
\usepackage{amsthm,amsmath,amssymb, amscd}
\usepackage{mathrsfs}
\usepackage{amsfonts}
\usepackage[utf8]{inputenc}
\usepackage{subfigure}
\usepackage{graphicx}
\usepackage{bbding}
\everymath{\displaystyle}
\usepackage{indentfirst}

\usepackage{enumerate}
\usepackage{bm,bbm}
\usepackage[shortlabels]{enumitem}
\usepackage{algorithm, algorithmic}
\usepackage{dsfont}
\usepackage{booktabs}
\usepackage{xspace}
\usepackage{pifont}
\usepackage{xcolor}
\usepackage{cleveref}
\usepackage{dsfont}
\usepackage{mathpazo}
\PassOptionsToPackage{compress, square,numbers}{natbib}
\usepackage{natbib}
\numberwithin{equation}{section}
\usepackage{apptools}
\AtAppendix{\counterwithin{thm}{section}}

\input{math_commands}

\input{def}

\renewcommand{\P}{\mathbb{P}}
\renewcommand{\eqref}[1]{(\ref{#1})}
\usepackage{colortbl}
\definecolor{Ocean}{RGB}{129,194,234}

\title{Actor-Critic Learning for Extended Mean Field Control with Deterministic Policies
}

\numberwithin{equation}{section}

\theoremstyle{plain}

\newtheorem{assumption}{Assumption}[section]

\theoremstyle{definition}

\theoremstyle{remark}

\numberwithin{table}{section}

\usepackage{bbm}

\def\cF{\mathcal{F}}

\def\cL{\mathcal{L}}

\def\cO{\mathcal{O}}
\def\cP{\mathcal{P}}

\def\cW{\mathcal{W}}

\def\d{{\mathrm{d}}}

\def\sE{{\mathbb{E}}}
\def\sF{{\mathbb{F}}}

\def\sP{\mathbb{P}}

\def\sR{{\mathbb R}}

\DeclareMathOperator*{\tr}{tr}

\DeclareMathOperator*{\id}{\textnormal{id}}

\begin{document}
\author{
Ziheng Cheng\thanks{University of California, Berkeley. Email: \texttt{ziheng\_cheng@berkeley.edu}}
\and Xin Guo\thanks{University of California, Berkeley. Email: \texttt{xinguo@berkeley.edu}}
\and Huy\^en Pham \thanks{Ecole Polytechnique, CMAP, Email: \texttt{huyen.pham@polytechnique.edu} }
\and Yufei Zhang\thanks{Imperial College London. Email: \texttt{yufei.zhang@imperial.ac.uk}}
} 

\date{}

\maketitle 
\begin{abstract}
This paper develops a model-free reinforcement learning framework for continuous-time extended mean field control problems, where both the dynamics and reward may depend on the joint distribution of states and controls.  We adopt deterministic feedback policies, under which the state–action distribution is induced directly as a push-forward of the state law. This avoids optimization over stochastic kernels and bypasses key limitations of existing approaches in extended mean field settings. 

We first establish a model-free sensitivity formula for parameterized McKean–Vlasov dynamics and use it to derive a deterministic policy gradient formula expressed through an advantage-rate function on the Wasserstein space. We then refine this formula by introducing local value and advantage-rate representations that depend on the state, action, and joint state–action distribution, yielding a policy gradient that includes both action derivatives and measure-derivative terms with respect to the control distribution. These characterizations lead to a martingale-based learning principle and motivate a continuous-time deep deterministic policy-gradient algorithm combining particle approximations, measure-dependent neural networks, temporal-difference learning, and exploration in either action or parameter space. 

Numerical experiments on stochastic Cucker–Smale consensus control and optimal liquidation with trade crowding demonstrate the efficiency, stability, and robustness of the proposed method, including  problems with explicit dependence on the control distribution.
 \end{abstract}

%

\section{Introduction}

\paragraph{(Extended) MFC.} 

(Extended) mean field control (MFC) provides a tractable framework for large population stochastic games, by optimizing the collective behavior  of homogeneous and interacting agents through a central planner \cite{acciaio2019extended}. In the mean-field limit, the controlled state of a (representative) agent is governed by a McKean-Vlasov system 
\begin{align}
 \label{eq:intro_MV}
\begin{split}
 \d X^{\alpha}_s =b(s,X^{\alpha}_s,
 {\alpha}_s, \sP_{(X^{\alpha}_s,
 {\alpha}_s)})\d s+\sigma(s,X^{\alpha}_s,
 {\alpha}_s, \sP_{(X^{\alpha}_s,
 {\alpha}_s)})\d W_s,
 \end{split} 
\end{align}
where both the drift and diffusion  coefficients depend on the   distribution of the state  $X^{\alpha}$ and the control ${\alpha}$.  Although this limit removes the explicit dependence on the number of agents, the associated extended MFC is intrinsically an infinite-dimensional control problem: the value function and the optimal feedback generally depend not only on time and the individual state, but also on the state distribution. 
The most fundamental theoretical ingredients in the development of MFC theory are the dynamic programming principle for the lifted value function defined on the space of probability measures, and the invariance property of this value function along the deterministic state flow (see e.g., \cite{pham2018bellman, cosso2023optimal}).

\paragraph{Continuous-time RL with deterministic policy.}

In extended mean field control, both the dynamics and reward may depend on the joint distribution of states and controls.
To address this challenge, we adopt deterministic feedback policies.
A deterministic policy $\varphi(t,x,\mu)$
selects an action directly based on the current time, state and state law,  
so that the state law together with the feedback map completely determines the associated state--action distribution through the push-forward measure  $\Gamma_t^{\varphi}=
\bigl(\operatorname{id},\varphi(t,\cdot,\mu_t)\bigr)_{\#}\mu_t,
$ with $\mu_t$ being the current state distribution. 
For continuous-time reinforcement learning (RL), 
recent analysis using  deterministic   policies  \cite{cheng2025deterministic,guo2026deterministic} have shown superiority  over both the discretization approach and the stochastic policies. Analytically,  the state–action distribution is induced
directly as a push-forward of the state law.
Hence the policy optimization  over controls is reduced to an optimization over a parameterized family of  deterministic policies  $\varphi_{\theta}$ (with respect to parameter $\theta$). 
Numerically, deterministic policies have also been shown to offer improved computational efficiency and greater stability \cite{cheng2025deterministic}.

\paragraph{Learning for discrete-time MFC.}
In the RL  framework of MFCs, the agent does not know the state coefficients $b$ and $\sigma$, nor the reward functions, and instead, learns by interacting with the controlled system. In discrete time, dynamic programming principles have been established for both state- and state-action value functions, leading to value-based algorithms including Q-learning and integrated Q-functions based on local policies  \cite{motte2022mean, carmona2023model, gu2023dynamic}. 
More recently, model-free policy-gradient methods have also been proposed in \cite{meunier2026MFC}, culminating in the MF-REINFORCE algorithm with theoretical guarantees on the bias and variance of the policy-gradient estimator.

\paragraph{Our work for continuous-time extended MFC.}
 This paper develops an RL framework for conti\-nuous-time extended MFC, adopting the deterministic policy approach.
 

We first develop a model-free deterministic policy gradient formula for continuous-time extended MFC    (Theorem \ref{thm:dpg_flow}). This representation expresses the gradient of value function with respect to the policy parameter  as an integral involving the derivative of an advantage-rate function along the flow of state distributions. This advantage-rate function is lifted to a function on the space of probability measures and is characterized jointly with the value function through an invariance principle along the  flow of state distributions.  These results are derived using a model-free sensitivity formula for parametric McKean–Vlasov dynamics (Theorems \ref{thm:sensitivity} and \ref{thm:sensitivity_critic}), which relies on a performance difference lemma 
(Lemma \ref{prop:performance_difference})
and the robustness of the state distribution with respect to model parameters. These results for continuous-time extended MFC are consistent with their discrete-time counterpart  regarding lifted (state–action) value functions and Q-functions \cite{motte2022mean, carmona2023model, gu2023dynamic}.

We then refine the deterministic policy gradient formula by explicitly accounting for its dependence on the joint state–action distribution (Theorem \ref{thm:martingale}). Specifically, we decompose the lifted value function and the lifted advantage-rate function into integrals of local functions of the state, action, and the joint state–action distribution. The resulting policy gradient contains both the usual derivative of the local advantage-rate function with respect to the individual action and an additional measure-derivative term with respect to control distribution. We further establish a martingale characterization to learn  these local representations directly from observed state  and control trajectories.
 Just as the incorporation of local policies facilitates efficient learning in the 
discrete-time MFC \cite{gu2023dynamic}, this local advantage-rate function   allows more informative learning  
 than directly learning 
from the flow of state laws (Section \ref{sec:alg}).
This approach applies to general extended MFC problems, and removes the restrictive assumption in   \cite[Remark 3.6]{pham2025actor} which requires a separable structure on the coefficients and explicit and known dependence on the control variable.

Finally, we propose a continuous-time deep deterministic policy-gradient algorithm for extended MFC. The algorithm combines measure-dependent neural networks, particle approximations of the McKean-Vlasov dynamics, temporal-difference learning based on the aforementioned martingale characterization, and action- or parameter-space exploration. The use of finite-dimensional distribution embeddings is consistent with the neural approximation principles developed in \cite{mekkaoui2026learning, soner2025learning}, while the deterministic actor avoids direct optimization over a space of stochastic kernels. Numerical experiments for consensus control in a stochastic Cucker-Smale system and optimal liquidation with mean-field market impact demonstrate the efficiency and robustness of the proposed method, including in settings where the dynamics depend explicitly on the distribution of the control.

\paragraph{Earlier work for RL in continuous-time-space MFC.}  Many earlier studies on continuous-time MFC reinforcement-learning have been built on stochastic, or exploratory, policies; and focuses on the setting  where the coefficients depend only on the law of the state process. In this exploratory formulation,  the admissible policy class is enlarged to stochastic policies that map the time, state, and state distribution to a probability measure over the action space.
However, RL algorithms under the exploratory control formulation appear too restrictive for
the infinite-dimensional nature of MFC.    For example, in \cite{pham2025actor}, the state coefficients are required to possess a separable structure, and the explicit dependence on the control variable must be known in order to characterize the mean-field $q$-function. In \cite{wei2024unified,wei2025continuous},  algorithms require learning   two $q$-functions: an integrated $q$-function identified via weak martingale conditions by testing against all stochastic policies, and an essential $q$-function used for policy improvement. In \cite{ren2026continuousI}, the exploratory HJB equation contains an additional nonlinear functional of the policy, and the optimal policy is characterized by a two-layer fixed-point problem over the space of stochastic policies.  The corresponding algorithms in  \cite{ren2026continuousII} then face the issue that data generated by the relaxed-control formulation are not directly observable, requiring additional approximation by discretely sampled exploratory actions, as in \cite{szpruch2024optimal,jia2026accuracy}.

 To see analytically the fundamental obstacle  of stochastic-policy methods for learning extended MFC,  consider a stochastic Markov policy
$\pi(\mathrm{d}a | t,x,\mu)$
 and its induced state-action distribution
    $\Gamma_t^{\pi}(\mathrm{d}x,\mathrm{d}a)
    = 
    \bigl(\operatorname{id},
    \pi(\mathrm{d}a | t,\cdot,\mu_t)
\bigr)_{\#}\mu_t$.
In a standard MFC problem whose coefficients depend {\it only on the state distribution}, stochastic policies produce an auxiliary generator by integrating the original state generator over the action space with respect to  the stochastic policy $\pi$, which is linear in the stochastic policy. This linearity is central for stochastic-policy based algorithms  \cite{frikha2025actor, pham2025actor, bayraktar2026policy}.
When applying the  exploratory control formulation   to extended MFC, randomizing actions through relaxed controls necessarily alters the control distribution $\sP_{\alpha_t}$ in \eqref{eq:intro_MV}, so that the action randomization can no longer be separated from the interaction induced by the control distribution. Thus, one cannot derive an auxiliary dynamics whose generator depends linearly on the stochastic policy.  

In fact, stochastic policies face  algorithmic challenges even in the single-agent setting. They require sampling random actions at very high frequency, leading to irregular control trajectories that may be impractical in real-world applications \cite{szpruch2024optimal,jia2026accuracy,ren2026continuousII}. Moreover, the Bellman equation under stochastic policies  involves integration over continuous action spaces. Enforcing this condition requires   Monte Carlo approximation, which is computationally expensive and often leads to unstable and slow convergence.

In contrast, 
the benefit  of deterministic policies for MFC is clear: the state–action distribution is induced directly as a push-forward of the state law. It  avoids optimization over stochastic kernels, helps bypass key limitations of exploratory-policy approaches in extended mean field settings, and facilitates policy updates by differentiating through both the selected actions and the induced state--action distributions. As demonstrated in \cite{cheng2025deterministic} for single-agent RL problems and the numerical experiments in this paper, deterministic policies offer significant computational advantages   in continuous-time-space settings, leading to model-free algorithms with improved stability and faster convergence compared with stochastic-policy methods.

\paragraph{Notation.}

We denote by $x\cdot y$ the scalar product between two vectors $x\in \sR^m$ and $y\in \sR^m$,
and by $M\colon N=\tr(M^\top N)$ the inner product between 
two matrices $M\in \sR^{m\times n}$ and $ N\in  \sR^{m\times n}$.

Given $T>0$, a filtered probability space $(\Omega, \mathcal{F}, \mathbb{F}=(\mathcal F_t)_{t\in [0,T]},  \mathbb{P})$, 
 a normed space $(E, |\cdot|)$, 
and an $E$-valued random variable $X$ on $(\Omega, \mathcal{F},    \mathbb{P})$, we denote by $\mathbb{P}_X$ its probability law under $\mathbb{P}$,
by $X\sim \mu$ the fact that  $X $ follows the distribution $\mu$,  by 
$\sE_{\xi\sim \mu}[f(\xi)]=\int_E f(x)\mu (\d x)$
for a measurable  function $f:E\to \sR$,
and 
by $L^2(\mathcal{F}_t; E)$   the set of square integrable   $E$-valued random variables on $(\Omega, \cF_t,\sP)$ for all $t\in [0,T]$.  
We denote by $\mathcal{P}_2(E)$ the set of probability measures $\mu$ on $E$ that are square integrable, i.e.,
$M_2(\mu) \coloneqq \left(\int_E |x|^2 \mu(dx)\right)^{1/2} < \infty.
$
We shall assume without loss of generality   that $\mathcal{F}_0$ is rich enough to carry $E$-valued random variables with any   square integrable distribution, i.e.,
$\mathcal{P}_2(E) = \{ \mathbb{P}_\xi \mid \xi \in L^2(\mathcal{F}_0; E) \}.
$
 We   equip $ \mathcal{P}_2(E) $ with  the 2-Wasserstein distance $ \cW_2 $  defined    by
\begin{align*}
\cW_2(\mu, \mu') &= \inf \left\{ \left( \int_{E \times E} |x - y|^2 \, \pi(dx, dy) \right)^{1/2} \,\bigg\vert\, \pi \in \mathcal{P}_2(E \times E) \text{ with marginals } \mu \text{ and } \mu' \right\}
\\
&= \inf \left\{ \left( \mathbb{E} \left[ |\xi - \xi'|^2 \right] \right)^{1/2} \,\Big\vert\, \xi, \xi' \in L^2(\mathcal{F}_0; E), \, \mathbb{P}_\xi = \mu, \, \mathbb{P}_{\xi'} = \mu' \right\}.
\end{align*}

We  denote by  
$ C^{1,2}([0, T]\times  \mathcal{P}_2(\mathbb{R}^n)) $   
the space of continuous functions  
$ w : [0, T] \times  \mathcal{P}_2(\mathbb{R}^n) \to \mathbb{R} $ such that 
for all $\mu\in  \cP_2(\sR^n)$,
 $  w(\cdot, \mu)$ is $C^{1}([0,T])$, with   $\partial_t w$  being  continuous in all variables; 
for all $t\in [0,T]$,     $w(t,\cdot)$ is  L-differentiable, 
and $(t,\mu,v)\mapsto \partial_\mu w(t,\mu)(v)$ is continuous at all $(t,\mu,v)$ with   $v\in \operatorname{supp}(\mu)$;  
for all $(t,\mu)\in [0,T]\times \cP_2(\sR^n) $,  $  \partial_\mu w(t,\mu)(\cdot)$ is $C^1$, with 
$(t,\mu, v) \mapsto \partial_v \partial_\mu w(t,\mu)(v)$ being continuous at all $(t,\mu,v)$ with   $v\in \operatorname{supp}(\mu)$; and  for all  compact sets $\mathcal K\subset   \cP_2(\sR^n)$,
$$
\sup_{(t,\mu) \in [0,T]\times \mathcal K}\left[\int_{\sR^n} |\partial_\mu w(t, \mu) (\xi)|^2\mu (\d \xi)
+\|\partial_v \partial_\mu w(t,\mu)\|_{L^\infty_\mu} \right]<\infty. 
$$

\section{Model-free sensitivity formula of McKean-Vlasov dynamics} 
\label{sec:sensitivity}

This   section considers a multidimensional McKean--Vlasov dynamics parameterized by $\theta$ and 
derives  model-free representations  of the gradient of the value functional with respect to $\theta$. 
{
Consistent with the existing literature on MFCs, we consider the lifted value function defined on the Wasserstein space of probability measures. The derivation of these representations exploits both the invariance property of the value function and the robustness of the state distribution with respect to the model parameters. 
These representations will be useful for developing   learning algorithms for the controlled McKean-Vlasov dynamics in Section~\ref{sec:DPG_MFC}.}

Let $T>0$   be a given terminal time, and $(\Omega, \mathcal{F}, \mathbb{P})$ be  a  probability space on which an $m$-dimensional Brownian motion $W = (W_t)_{0 \leq t \leq T}$ is defined. 
We   denote by $\mathbb{F} = (\mathcal{F}_t)_{0 \leq t \leq T}$ the filtration generated by $W$, augmented with $\sP$-null sets, and assume that  the initial  $\sigma$-algebra $\mathcal{F}_0 $ is      rich  enough.
Let 
$\beta\ge 0$
be a discount factor, and consider continuous functions 
$b:[0,T]\times \sR^n\times \cP_2(
\sR^n)\times \sR^k\to \sR^n$, 
$\sigma:[0,T]\times \sR^n\times \cP_2(
\sR^n)\times \sR^k\to \sR^{n\times m}$,
$r:[0,T]\times \sR^n\times \cP_2(
\sR^n)\times \sR^k\to \sR$
and 
$g:\sR^n\times \cP_2(\sR^n)\to \sR$ 
satisfying the following regularity   and growth conditions:
\begin{assumption}
\label{assum:continuity}
There exists a locally bounded function $\omega:[0,\infty)\to [0,\infty) $  such that for all $(t,\theta)\in [0,T]\times \sR^k$, 
$(x,x')\in \sR^n$ and $\mu,\mu'\in \cP_2(\sR^n)$, 
\begin{align*}
|b(t,x,\mu,\theta)-b(t,x',\mu',\theta)|+|\sigma(t,x,\mu,\theta)-\sigma(t,x',\mu',\theta)|
&\le \omega(|\theta|)(|x-x'|+\cW_2(\mu, \mu')),
\\
|b(t,0,\delta_0,\theta) |+|\sigma(t,0,\delta_0,\theta)|
+\frac{|r(t,x,\mu,\theta) |+|g(x,\mu)|}
{1+|x|^2+M_2( \mu)^2}
&\le \omega(|\theta|).
\end{align*}

 \end{assumption}

For each
$t\in [0,T]$,
 $\theta\in \sR^k$, 
 and  $ \xi\in  L^2(\mathcal{F}_t; \sR^n)$, 
  consider 
  the following state dynamics:
 \begin{align}
\d X^{t, \xi,\theta}_s &=b(s,X^{t, \xi, \theta}_s, \sP_{X^{t, \xi,\theta}_s}, \theta)\d s+\sigma(s,X^{t, \xi,\theta}_s, \sP_{X^{t, \xi, \theta}_s}, \theta)\d W_s,
\quad s\in [t,T];
  \quad X^{t, \xi,\theta}_t=\xi.
\label{eq:sde_mv_population}
\end{align}
Under Assumption \ref{assum:continuity},  
\eqref{eq:sde_mv_population}
has a  
unique square-integrable strong solution  $(X^{t, \xi,\theta}_s)_{s\in [t,T]}$ (see e.g., \cite{pham2018bellman} and the references therein).
%
 Moreover, 
due to the  weak uniqueness of \eqref{eq:sde_mv_population},  the law of    $(X^{t, \xi,\theta}_s)_{s\in [t,T]}$  
 depends on  $\xi$ only through its law $\mu=\sP_{\xi}$,  
 and hence 
 we can define
 $$
 \sP^{t,\mu,\theta}_s\coloneqq \sP_{X^{t,\xi, \theta}_s},
 \quad 
s\in [t,T].
 $$
Define  the following value function  
for a given    $\beta\ge 0$:
\begin{align}
\label{eq:cost_theta}
V(t,\mu,\theta) &\coloneqq \sE \left[
\int_t^T 
e^{-\beta (s-t)}
r\left(s, X^{t,\xi,  \theta}_s, \sP_{{X^{t,\xi, \theta}_s}}, \theta\right) \d s +
e^{-\beta (T-t)}
g\left(X^{t,\xi, \theta}_T, \sP_{X^{t,\xi, \theta}_T}\right)
\right]
\\
&=
\int_t^T 
e^{-\beta (s-t)}
\bar r \left(s,  \sP^{t,\mu,\theta}_s, \theta\right) 
\d s +
e^{-\beta (T-t)}
 \bar g ( \sP^{t,\mu,\theta}_T ),
\end{align}
where the function
$\bar r: [0,T]\times \cP_2(\sR^n)\times \sR^k\to \sR$
 and 
 $\bar g:\cP_2(\sR^n)\to \sR$ are defined by  
\begin{equation}
  \bar r (t,\mu,\theta)\coloneqq 
  \int_{\sR^n}
  r(t,x,\mu,\theta) \mu(\d x), 
\quad 
\bar g(\mu ) \coloneqq  \int_{\sR^n}  g(x,\mu)\mu(\d x).
\end{equation}

To  characterize the gradient of $V$ in $\theta$,
for all  
   $w\in C^{1,2}([0,T]\times \cP_2(\sR^n))$,
define the function  
$A[w]:[0,T]\times \cP_2(\sR^n)\times \sR^k \to \sR$
such that for all $(t,\mu,\theta)\in [0,T]\times\cP_2(\sR^n)\times \sR^k$,
   \begin{equation}
   \label{eq:advantage_rate}
     A[w](t,\mu,\theta)\coloneqq  \mathcal L^{  \theta}[w] 
    (t, \mu )  +\bar r(t,\mu,\theta), 
  \end{equation}
  where 
  $ \cL^\theta$  is the generator of \eqref{eq:sde_mv_population} given by
  \begin{align}
\label{eq:generator}
\begin{split}
     \cL^\theta[ w](t, \mu)
     &\coloneqq 
     \partial_t w(t, \mu) -\beta  w(t, \mu)  
        \\
    &\quad+ \sE_{\xi\sim\mu} 
    \left[ b(t,\xi,\mu,\theta)\cdot \partial_\mu w(t, \mu) (
    \xi) 
    +\frac{1}{2}  (\sigma\sigma^\top) (t,\xi,\mu,\theta)\colon     
     \partial_v \partial_\mu w(t, \mu)(\xi) \right].
     \end{split}
\end{align}
We impose the following  regularity conditions  on the value function   $V^\theta$ and the function $A[w]$.

\begin{assumption}
\label{assum:smoothness}
  
\begin{enumerate}[(1)]
\item 
\label{item:value_smooth}
For all $\theta\in \sR^k$,
the   function    $
V(\cdot,\cdot,\theta)  $  
defined by \eqref{eq:cost_theta} 
is in  $C^{1,2}([0, T]    \times \mathcal{P}_2(\mathbb{R}^n))$.

\item 
\label{item:differentable_theta}

For all  
   $w\in C^{1,2}([0,T]\times \cP_2(\sR^n))$,
the function $A[w]$ defined by \eqref{eq:advantage_rate} is    differentiable with respect to $\theta$, 
   and the derivative  
   $\partial_\theta  A[w] : [0,T]\times \cP_2(\sR^n)\times \sR^k\to \sR^k $ is continuous. 
\end{enumerate}

\end{assumption}

Under Assumptions   \ref{assum:continuity}  and 
 \ref{assum:smoothness}, 
 the following theorem shows that the gradient $\partial_\theta V$
 can be expressed as the integral of  $\partial_\theta A$  in \eqref{eq:advantage_rate} along the flow of state laws.

\begin{thm}
\label{thm:sensitivity}
 Suppose Assumptions   \ref{assum:continuity}  and 
 \ref{assum:smoothness}
 hold.  
 For all $(t,  \mu)\in [0,T] \times \cP_2(\sR^n)$
 and $ \theta\in \sR^k$,
 \begin{align}
 \begin{split}
& \partial_\theta V(t,  \mu, \theta )
  =  
    \int_t^T
   e^{-\beta(s-t)} 
   \partial_\theta 
   A [V(\cdot,\cdot,\theta)] 
    (s, \sP^{t,\mu,\theta}_s ,\theta) \, 
   \d s. 
 \end{split}
 \end{align}
   
\end{thm}

 Theorem \ref{thm:sensitivity}
expresses  the gradient $\partial_\theta V$ directly   in terms of   $\partial_\theta 
   A [V] $, without reference to the model coefficients.
The proof leverages a precise  characterization of the difference between the value functions     corresponding to two different parameters  (Proposition \ref{prop:performance_difference}).
Such a result extends the performance-difference lemmas in \cite[Lemma 6.1]{cheng2025deterministic} and \cite[Lemma 3.2]{sethi2025entropy} from classical control problems to the more general setting of McKean--Vlasov dynamics.
 Compared with existing approaches in \cite{jia2022policy_grad,frikha2025actor}, our method avoids differentiating the PDE for the value function  $V(\cdot,\cdot,\theta)$ with respect to $\theta$,  
 and therefore requires weaker regularity assumptions on the coefficients. In particular, we do not require differentiability in $\theta$ of the first- and second-order time and spatial derivatives of
   $V$, which is imposed in  \cite[Appendix A2]{frikha2025actor}.

{To characterize the advantage rate function, we first show that the value function $V$ and the advantage rate function $A[V]$ must satisfy a Feynman-Kac formula and  an invariance property along the deterministic state flow.

 }

\begin{thm}
\label{thm:sensitivity_critic_ncessary}
 Suppose Assumptions   \ref{assum:continuity}  and 
 \ref{assum:smoothness}
 hold. 
Let  $\theta\in \sR^k$,
 $ V^\theta\coloneqq V(\cdot, \cdot,\theta) \in C^{1,2}([0,T]\times \cP_2(\sR^n)) $ be 
  defined by \eqref{eq:cost_theta}, and   $ q^\theta =  A [V^\theta]\in C ([0,T]\times \cP_2(\sR^n)\times \sR^k)$ be defined  by \eqref{eq:advantage_rate}. 
For all 
 $(t,\mu)\in [0,T]\times \cP_2(\sR^n)$, 
\begin{equation}
      \label{eq:bellman_param_necessary}  {V}^{\theta}(T,\mu)=\bar g(\mu),\quad  {q}^\theta(t,\mu,\theta)=0,
    \end{equation}
and for all $\theta'\in \sR^k $ and $s\in [t,T]$,
    \begin{equation}
   \label{eq:constant_flow_necessary}
        e^{-\beta s }{V}^{\theta}(s,\P^{t,\mu,\theta'}_s)
        - e^{-\beta t }  V^{\theta}(t,\mu)
        +\int_t^s e^{-\beta u}[\bar r (u,\sP^{t,\mu, \theta'}_{u},\theta')-{q}^\theta(u,\sP^{t,\mu, \theta'}_{u},\theta')]\,\d u = 0,
    \end{equation}
where   
$\sP^{t,\mu, \theta'}_{r}=\sP_{X^{t,\xi,\theta'}_r}$,
and 
$X^{t,\xi,\theta'}$ satisfies 
for all $s\in [t,T]$,
  \begin{align*}
\d X^{t, \xi,\theta'}_s &=b(s,X^{t, \xi, \theta'}_s, \sP_{X^{t, \xi,\theta'}_s}, \theta')\d s+\sigma(s,X^{t, \xi,\theta'}_s, \sP_{X^{t, \xi, \theta'}_s}, \theta')\d W_s,
  \quad X^{t, \xi,\theta'}_t=\xi,
\end{align*}
with some $\xi\in L^2(\cF_t;\sR^n)$ 
having the law 
  $ \mu$.

\end{thm}

Conversely, we show that the conditions \eqref{eq:bellman_param_necessary} and \eqref{eq:constant_flow_necessary} are sufficient to jointly characterize both the value function $V$ and the advantage rate function $A[V]$.
This derivation is based on two key properties in general MFC: the first being the invariance property of the value function for any given policy along the deterministic state flow in a neighborhood of  $\theta$, and the second being the continuity of  the 
 law of McKean–Vlasov systems with respect to the perturbations of the initial condition and model parameters.

\begin{thm}
\label{thm:sensitivity_critic}
 Suppose Assumptions   \ref{assum:continuity}  and 
 \ref{assum:smoothness}
 hold. 
Let $\theta\in \sR^k$,
  $\hat V\in C^{1,2}([0,T]\times \cP_2(\sR^n))$,
  and 
 $\hat q\in C([0,T]\times \cP_2(\sR^n)\times \sR^k)$ 
 satisfy the following conditions: for all 
 $(t,\mu)\in [0,T]\times \cP_2(\sR^n)$, 
\begin{equation}\label{eq:bellman_param}
        \hat{V}(T,\mu)=\bar g(\mu),\quad \hat{q}(t,\mu,\theta)=0,
    \end{equation}
and  there exists a neighborhood
$\cO_{t,\mu}(\theta)\subset \sR^k$
of $\theta$ such that for all $\theta'\in \cO_{t,\mu}(\theta) $ and $s\in [t,T]$,
    \begin{equation}
    \label{eq:constant_flow}
        e^{-\beta s }\hat{V}(s,\P^{t,\mu,\theta'}_s)
        - e^{-\beta t } \hat V(t,\mu)
        +\int_t^s e^{-\beta u}[\bar r (u,\sP^{t,\mu, \theta'}_{u},\theta')-\hat{q}(u,\sP^{t,\mu, \theta'}_{u},\theta')]\,\d u = 0,
    \end{equation}
where   
$(\sP^{t,\mu, \theta'}_{r})_{r\in [t,T]}$
is defined as in Theorem \ref{thm:sensitivity_critic_ncessary}. 
Then for all $(t,\mu,\theta')\in [0,T]\times \cP_2(\sR^n)\times  \cO_{t,\mu}(\theta) $,
$$
\hat V(t,\mu)=V(t,\mu,\theta), \quad \hat q(t,\mu,\theta')=A[V(\cdot, \cdot \theta)](t,\mu,\theta').
$$

\end{thm}

For any given $\theta$, Theorem \ref{thm:sensitivity_critic} characterizes the associated value function and advantage rate function in a neighborhood of $\theta$, via the Feynman–Kac formula \eqref{eq:bellman_param} and the invariance principle \eqref{eq:constant_flow} along deterministic state flows. This provides a unified framework for characterizing value and advantage functions under arbitrary policies, and applies to both deterministic policies as in Section \ref{sec:DPG_MFC} and stochastic policies as in \cite{pham2025actor,wei2025continuous}.

\section{Deterministic Policy Gradient for Extended MFC Problems}
\label{sec:DPG_MFC}

In this section, we apply the sensitivity formula in Section \ref{sec:sensitivity} 
to derive a policy gradient formula for
extended mean field control (MFC) problems
studied in \cite{acciaio2019extended,picarelli2025extended}.
In  extended MFC problems, both the cost functional and the state dynamics depend on the joint distribution of the controlled state and control processes. This formula provides the foundation for developing a model-free deterministic actor--critic algorithm in Section~\ref{sec:alg} for solving extended MFC problems.

\subsection{Problem formulation}

\paragraph{Extended MFC problems.}
Let $T>0$   be a given terminal time and $(\Omega, \mathcal F,\mathbb P)$   be a complete probability space which
supports an $m$-dimensional Brownian motion $W$ and an independent square-integrable random
variable $\xi_0$. We denote by $\mathbb F = (\mathcal F_t)_{t\ge 0}$  the filtration generated by $W$ and $\xi_0$ augmented by the
$\mathbb P$-null sets.
Let $A\subset \mathbb R^d$  be a measurable set representing the agent’s action space, and let $\mathcal U$ be
a space of all $\mathbb F$-adapted square-integrable processes $\alpha:[0,T]\times \Omega\to A$ representing the agent’s
admissible control space.

For each $\alpha\in \mathcal U$, consider the associated state process $X^\alpha$
governed by the following controlled McKean-Vlasov dynamics:
\begin{align}
\label{eq:sde_mv_open_loop}
\begin{split}
 \d X^{\alpha}_s =b(s,X^{\alpha}_s,
 {\alpha}_s, \sP_{(X^{\alpha}_s,
 {\alpha}_s)})\d s+\sigma(s,X^{\alpha}_s,
 {\alpha}_s, \sP_{(X^{\alpha}_s,
 {\alpha}_s)})\d W_s,
  \quad X^{\alpha}_0=\xi_0,
 \end{split} 
\end{align}
where 
$b:[0,T]\times \sR^n 
\times \sR^d \times \cP_2(\sR^n\times \sR^d)\to \sR^n$
and 
$\sigma:[0,T]\times \sR^n 
\times \sR^d \times \cP_2(\sR^n\times \sR^d)\to    \sR^{n\times m}$ are sufficiently regular  functions such that 
\eqref{eq:sde_mv_open_loop} has a unique square-integrable strong solution $X^\alpha$.
The   agent aims to maximize the following
reward functional
\begin{align}
\label{eq:reward_mv_open_loop}
J(\alpha)  \coloneqq \sE \left[
\int_0^T 
e^{-\beta s}
r (s,X^{\alpha}_s,
 {\alpha}_s, \sP_{(X^{\alpha}_s,
 {\alpha}_s)})  \d s +
e^{-\beta T}
g(X^{\alpha}_T, \sP_{X^{\alpha}_T})
\right]
\end{align}
over all $\alpha\in \mathcal U$,
where 
$\beta \ge 0$, 
 $r:
[0,T]\times \sR^n 
\times \sR^d \times \cP_2(\sR^n\times \sR^d)\to \sR$ and 
$g:\sR^n 
  \times \cP_2(\sR^n)\to \sR$
are continuous functions with at most quadratic growth. 

It is known that under suitable regularity conditions (see e.g., \cite{pham2018bellman}), it suffices to optimize \eqref{eq:reward_mv_open_loop}
over   control
processes $\alpha^\varphi=(\alpha^\varphi_t)_{t\in [0,T]}$ given in closed loop (or feedback) form:
\begin{equation}
\label{eq:alpha_feedback}
\alpha^\varphi_t= \varphi(t,X^\varphi_t, \sP_{X^\varphi_t}), \quad t\in [0,T],
\end{equation}
where 
$\varphi: [0,T]\times \sR^n\times \cP_2(\sR^n)\to A$
is a measurable function, called a Markov policy, and  
$X^\varphi$ is a solution to the following controlled McKean-Vlasov dynamics:
\begin{align}
\label{eq:sde_mv_cl}
\begin{split}
 \d X^{\varphi}_s =& b(s,X^{\varphi}_s,
 \varphi(s,X^\varphi_s, \sP_{X^\varphi_s}), \sP_{(X^{\varphi}_s,
 \varphi(s,X^\varphi_s, \sP_{X^\varphi_s}))})\d s
 \\
 &+\sigma(s,X^{\varphi}_s,
\varphi(s,X^\varphi_s, \sP_{X^\varphi_s}), \sP_{(X^{\varphi}_s,
 \varphi(s,X^\varphi_s, \sP_{X^\varphi_s}))})\d W_s,
  \quad X^{\varphi}_0=\xi_0,
 \end{split} 
\end{align}
Consequently, the goal of the agent is to maximize the following objective
\begin{align}
\label{eq:reward_mv_cl}
J(\varphi)  \coloneqq \sE \left[
\int_0^T 
e^{-\beta s}
r (s,X^{\varphi}_s,
 \varphi(s,X^\varphi_s, \sP_{X^\varphi_s}), \sP_{(X^{\varphi}_s,
 \varphi(s,X^\varphi_s, \sP_{X^\varphi_s}))})  \d s +
e^{-\beta T}
g(X^{\varphi}_T, \sP_{X^{\varphi}_T})
\right]
\end{align}
over all  admissible Markov policies $\varphi$.

\paragraph{RL with  deterministic policies.}

In the RL framework, the agent does not know the coefficients $b,\sigma, r$ and $g$. Instead, the agent interacts directly with the system \eqref{eq:sde_mv_cl} with different actions and updates her policies based on the observed state and reward trajectories.  

As in the classical McKean-Vlasov control framework, we restrict attention   to suitable Markov policies $\varphi:[0,T]\times \sR^n\times \cP_2(\sR^n)\to A$  for optimizing \eqref{eq:reward_mv_cl}. In the RL literature, such policies are referred to as deterministic policies, since they map each time-state-measure triple directly to an action.
{
A deterministic feedback policy $\varphi(t,x, \mu)$ 
selects an action directly
based on the time $t$, state $x$ and   law $\mu$, and the associated state--action distribution is the push-forward measure $
    \Gamma_t^{\varphi}
    = \bigl(\operatorname{id},\varphi(t,\cdot,\mu)\bigr)_{\#}\mu$ defined in \eqref{eq:push_forward}.
The joint distribution is therefore completely determined by the state law and the feedback map. Although the resulting dynamics may still depend nonlinearly on $\Gamma_t^{\varphi}$, policy optimization can be reduced  to an optimization over  a parameterized family $\varphi_{\theta}$.} 

More precisely,
let $\operatorname{Lip}([0,T]\times \sR^n\times \cP_2(\sR^n); A)$
be   the set of  continuous functions $\varphi:[0,T]\times \sR^n\times \cP_2(\sR^n)\to A$ that are Lipschitz continuous in $(x,\mu)$, uniformly on $t\in [0,T]$. 
Given a class of parameterized policies $ \mathscr P_\Theta
\coloneqq \{\varphi_\theta \in \operatorname{Lip}([0,T]\times \sR^n\times \cP_2(\sR^n); A)\mid  \theta\in \sR^k\}$,  we maximize the following functional over $\theta\in \sR^k$:
\begin{align}
\label{eq:dpg_cost_theta}
  J(\theta)  \coloneqq \sE \left[
\int_0^T 
e^{-\beta s}
r (s,X^{\varphi_\theta}_s,
 {\varphi_\theta}(s,X^{\varphi_\theta}_s, \sP_{X^{\varphi_\theta}_s}), \sP_{(X^{\varphi_\theta}_s,
 {\varphi_\theta}(s,X^{\varphi_\theta}_s, \sP_{X^{\varphi_\theta}_s}))})  \d s +
e^{-\beta T}
g(X^{\varphi_\theta}_T, \sP_{X^{\varphi_\theta}_T})
\right],
\end{align}
where we write $J(\theta) = J(\varphi_\theta)$ with a slight  abuse 
of notation.

In the sequel, we propose model-free deterministic policy gradient (DPG) algorithms that 
optimize \eqref{eq:dpg_cost_theta} via a gradient ascent algorithm.

\subsection{Deterministic policy gradient formula}

\paragraph{Reformulation as parametric McKean-Vlasov dynamics.}
 To compute    $\nabla_\theta J(\varphi_\theta)$ 
 using  Theorems \ref
{thm:sensitivity} and 
\ref{thm:sensitivity_critic},
we rewrite  \eqref{eq:dpg_cost_theta}  as a special case of the parametric McKean-Vlasov dynamics analyzed  in Section \ref{sec:sensitivity}. 
As in \cite{pham2018bellman},
 for each policy $\varphi \in \mathscr P_\Theta$
and $(t,\mu)\in [0,T]\times \cP_2(\sR^n)$,
let   $(\id, \varphi (t,\cdot, \mu))_\sharp \mu\in \cP_2(  \sR^n\times \sR^d)$
be   the  induced  state-action   measure given by  
\begin{align}
\label{eq:push_forward}
     [(\id, \varphi (t,\cdot, \mu))_\sharp \mu ](B)&\coloneqq 
     \mu(\{x\in \sR^n\mid (x,\varphi (t,x,\mu))\in B\}), \quad \forall B\in \mathcal B(\sR^n\times \sR^d).
\end{align}
Note that for any   control $\alpha^\varphi$ of the form \eqref{eq:alpha_feedback},   the joint state-control distribution satisfies 
$$\sP_{(X^{\varphi}_t,
 {\alpha}^\varphi_t)}
 =(\id, \varphi (t,\cdot, \sP_{X^{\varphi}_t}))_\sharp \sP_{X^{\varphi}_t}, \quad t\in [0,T].
 $$
Define  
the following controlled  coefficients associated with 
$   \mathscr P_\Theta $: 
for all  $\ell\in \{b,\sigma, r\}$,
\begin{align}
\label{eq:coefficient_mfc_dpg}
    \ell^\varphi (t, x, \mu,\theta) &\coloneqq \ell (t,x,  \varphi_\theta(t,x,\mu),  (\id, \varphi_\theta(t,\cdot, \mu))_\sharp \mu ).
\end{align}
The  dynamics \eqref{eq:sde_mv_cl} with $\varphi=\varphi_\theta$ can then be rewritten as
\begin{align}
\label{eq:dpg_state_psi_theta}
\begin{split}
 \d X_s =& b^{\varphi}(s,X_s,\sP_{X_s}, \theta) \d s
 +\sigma^{\varphi}(s,X_s,\sP_{X_s}, \theta)  \d W_s.
 \end{split} 
\end{align}
Consider   the dynamic version of  \eqref{eq:dpg_cost_theta} defined as follows (cf.~\eqref{eq:cost_theta})
:  for all $(t,\mu)\in [0,T]\times \cP_2(\sR^n)$,
let $\xi\in L^2(\cF_t;\sR^n)$ with $\sP_{\xi}=\mu$ and   
$X^{t,\xi,\theta}$ satisfy  \eqref{eq:dpg_state_psi_theta} on $ [t,T]$ with 
$X^{t,\xi,\theta}_t=\xi$, and define 
\begin{align}
\label{eq:V_varphi_theta}
V^\varphi(t,\mu,\theta) &\coloneqq 
\int_t^T 
e^{-\beta (s-t)}
\bar r^\varphi \left(s,  \sP^{t,\mu,\theta}_s, \theta\right) 
\d s +
e^{-\beta (T-t)}
 \bar g ( \sP^{t,\mu,\theta}_T ),
\end{align}
where 
$\sP^{t,\mu,\theta}_s =\sP_{X^{t,\xi,\theta}_s}$, 
 and  
\begin{equation}
    \bar{r}^\varphi(t,\mu,\theta)\coloneqq \int_{\sR^n}  r^\varphi (t, x, \mu,\theta) \mu(\d x),\quad \bar{g}(\mu)\coloneqq \int_{\sR^n} g(x,\mu)\mu(\d x).
\end{equation}
For each 
$w\in C^{1,2}([0,T]\times \cP_2(\sR^n))$, define 
\begin{equation}
\label{eq:advantage_rate_varphi_theta}
A^\varphi[w](t,\mu,\theta) = \mathcal L^\varphi[w](t,\mu,\theta)+\bar{r}^\varphi(t,\mu,\theta),
\end{equation}
where $\mathcal L^\varphi$ is the generator of \eqref{eq:dpg_state_psi_theta} and is defined as in \eqref{eq:generator}
 with $(b,\sigma)=(b^\varphi,\sigma^\varphi)$. 
The function $A^\varphi[w]$
can be viewed as a mean-field analogue of the advantage rate function in continuous-time RL problems driven by classical diffusion processes \cite{cheng2025deterministic}.

In the sequel, we  assume that the model coefficients $  \{b,\sigma, r ,g \} $  and the policies in   $   \mathscr P_\Theta $  are sufficiently regular
such that 
the induced coefficients given in  \eqref{eq:coefficient_mfc_dpg} 
 satisfy 
Assumption 
\ref{assum:continuity},
and the associated   value and advantage rate functions 
satisfy 
Assumption  \ref{assum:smoothness}.

 \begin{assumption}
\label{assum:regularity_mfc}
 
   The functions  $  \{b^\varphi,\sigma^\varphi, r^\varphi ,g \} $  
    satisfy Assumption \ref{assum:continuity},
    and $V^\varphi$ in \eqref{eq:V_varphi_theta}  and $A^\varphi$  in \eqref{eq:advantage_rate_varphi_theta} satisfy Assumption \ref{assum:smoothness}.
    
\end{assumption}

\paragraph{Characterizations of DPG.}

We now present   the first characterization of the deterministic policy gradient (DPG),
which  
 combines Theorems \ref{thm:sensitivity} and \ref{thm:sensitivity_critic}.

\begin{thm}\label{thm:dpg_flow}
Suppose Assumption \ref{assum:regularity_mfc} holds, and let  $\varphi_\theta\in \mathscr P_\Theta$ for a  given $\theta\in \sR^k$. 
Let 
   $\hat V\in C^{1,2}([0,T]\times \cP_2(\sR^n))$ 
  and 
 $\hat q\in C([0,T]\times \cP_2(\sR^n)\times \sR^k)$ 
 satisfy the following conditions:
 for all 
 $(t,\mu)\in [0,T]\times \cP_2(\sR^n)$,
 \begin{enumerate}[(i)]
     \item 
 
$   \hat{V}(T,\mu)=\bar g(\mu)$ and   $\hat{q}(t,\mu,\theta)=0 $,   
\item         
   there exists a neighborhood
$\cO_{t,\mu}(\theta)\subset \sR^k$
of $\theta$ such that for all $\theta'\in \cO_{t,\mu}(\theta) $ and $s\in [t,T]$,
    \begin{equation}
    \label{eq:constant_flow_policy}
        e^{-\beta s }\hat{V}(s,\P^{t,\mu,\theta'}_s)
        - e^{-\beta t } \hat V(t,\mu)
        +\int_t^s e^{-\beta u}[\bar{r}^\varphi (u,\sP^{t,\mu, \theta'}_{u},\theta')-\hat{q}(u,\sP^{t,\mu, \theta'}_{u},\theta')]\,\d u = 0,
    \end{equation}
where   
$\sP^{t,\mu, \theta'}_{r}=\sP_{X^{t,\xi,\theta'}_r}$,
and 
$X^{t,\xi,\theta'}$ satisfies 
for all $s\in [t,T]$,
  \begin{align}
  \label{eq:sde_mv_perturbed}
\d X^{t, \xi,\theta'}_s &=b^\varphi (s,X^{t, \xi, \theta'}_s, \sP_{X^{t, \xi,\theta'}_s}, \theta')\d s+\sigma^\varphi(s,X^{t, \xi,\theta'}_s, \sP_{X^{t, \xi, \theta'}_s}, \theta')\d W_s,
  \quad X^{t, \xi,\theta'}_t=\xi,
\end{align}
with some $\xi\in L^2(\cF_t;\sR^n)$ having the law $ \mu$.
\end{enumerate}
Then for all $(t,\mu,\theta')\in [0,T]\times \cP_2(\sR^n)\times  \cO_{t,\mu}(\theta) $,
$$\hat V(t,\mu)=V^\varphi(t,\mu,\theta), \quad  
 \hat q(t,\mu,\theta')=A^\varphi[V^\varphi(\cdot,\theta)](t,\mu,\theta').
 $$
Moreover, 
let $\sP^\theta_{s}=\sP_{X^{\varphi_\theta}_s}$ with $X^{\varphi_\theta}$
defined by \eqref{eq:sde_mv_cl}, it holds that 
\begin{align}
 \begin{split}
& \nabla_\theta J(\varphi_\theta) 
  =  
    \int_0^T
   e^{-\beta s} 
   (\partial_\theta 
   \hat q)   
    (s, \sP^{ \theta}_s ,\theta) \, 
   \d s. 
 \end{split}
 \end{align}

\end{thm}

Theorem \ref{thm:dpg_flow} expresses the DPG in terms of the gradient of the advantage rate function with respect to the policy parameter $\theta$.
Both the value function and the advantage rate function are lifted to functions on the space of probability measures, and Theorem \ref{thm:dpg_flow} simultaneously characterizes them through an invariance principle along the flow of the state process  \eqref{eq:sde_mv_cl} controlled by $\varphi_\theta$.

{
{This theorem  is the continuous analogue of  \cite{motte2022mean, gu2023dynamic} for discrete-time extended MFC problems. In particular, \cite{gu2023dynamic} shows that 
the classical Q-function defined on the original action space cannot be characterized via the dynamic programming principle (DPP). To ensure the DPP,  one instead needs to construct a lifted Q-function called IQ-function on the space of policies defined by integrating the classical Q-function over the state--action distribution induced by a given policy}.
}

The next theorem refines Theorem~\ref{thm:dpg_flow} by decomposing the value function and the advantage rate function into integrals of local functions of the state, control, and their associated measures, and by characterizing these local functions through suitable martingale conditions.
We  impose the following regularity assumptions on these local functions:

\begin{assumption}
\label{assum:V_D_q_D}
Let 
   $\hat V_D\in C([0,T]\times \sR^n \times \cP_2(\sR^n))$ 
 be such that  $\hat V: [0,T]\times \cP_2(\sR^n)\to \sR$ given by
\begin{align} 
\label{eq:decompose_V}
 \hat V(t,\mu)\coloneqq \sE_{\xi\sim\mu}[ \hat V_D(t,\xi,\mu)] 
 \end{align}
  is well-defined and in the space 
 $C^{1,2}([0,T]\times \cP_2(\sR^n))$.
  Let  
 $\hat q\in C([0,T]\times \sR^n\times \sR^d\times \cP_2(\sR^n\times \sR^d))$
 and  the policy class $\mathscr P_\Theta $  be such  that  
$\hat q : [0,T]\times \cP_2(\sR^n)\times \sR^k \to \sR$ given by
\begin{align}
\label{eq:decompose_q}
 \hat q (t,\mu,\theta)\coloneqq \sE_{\xi\sim \mu}[\hat{q}_D(t,\xi,\varphi_{\theta}(t,\xi,\mu),(\id, \varphi_{\theta} (t,\cdot, \mu))_\sharp \mu)] 
\end{align}
   is well-defined and  continuous.  Moreover, $\partial_\theta \hat q$ exists,   is continuous, and satisfies the chain rule: 
 \begin{align}
 \label{eq:chain_rule_q}
 \begin{split}
   (\partial_\theta 
   \hat q)   
    (t, \mu,\theta) 
    &=
  \sE\bigg[  \partial_\theta\varphi_\theta(t,\xi,\mu)^\top 
 \Big\{   (\partial_a 
   \hat q_D)   
    (t, \xi,
\varphi_\theta(t, \xi,\mu),
 (\id, \varphi_{\theta} (t,\cdot, \mu))_\sharp \mu
) 
\\
&\quad +
\widetilde{\sE}\left[(\partial_\nu 
   \hat q_D)   
    \left(t, \widetilde{\xi}, 
\varphi_\theta(t,\widetilde{\xi},\mu),
 (\id, \varphi_{\theta} (t,\cdot, \mu))_\sharp \mu
\right)\right] (\xi,
\varphi_{\theta} (t,\xi, \mu)
) \Big\}\bigg],
\end{split}
 \end{align}
 where $\xi$ and $\widetilde{\xi}$ are    independent random variables with law $\mu$,
 $\widetilde{\sE}$ denotes the expectation with respect to $\widetilde{\xi}$,
and
$\partial_\nu \hat q_D(\cdot, \Gamma)$ is the partial  L-derivative of $\hat q_D$ with
respect to the second marginal of $\Gamma\in \cP_2(\sR^n\times \sR^d)$.

\end{assumption}

\begin{thm}\label{thm:martingale}

Suppose Assumption \ref{assum:regularity_mfc} holds, and  
   $\hat V_D\in C([0,T]\times \sR^n \times \cP_2(\sR^n))$,
 $\hat q_D\in C([0,T]\times \sR^n\times \sR^d\times \cP_2(\sR^n\times \sR^d))$ 
 and  the policy class $\mathscr P_\Theta $  
 satisfy Assumption \ref{assum:V_D_q_D}. 
Let  $\theta\in \sR^k$, and assume that $\hat V_D$ and $\hat q_D$ 
 satisfy the following conditions:
\begin{enumerate}[(i)]
\item
 For all 
 $(t,x,\mu)\in [0,T]\times \sR^n\times \cP_2(\sR^n)$,
\begin{equation}\label{eq:bellman}
    \hat{V}_D(T,x,\mu)=g(x,\mu),
\quad 
\sE_{\xi\sim \mu}[\hat{q}_D(t,\xi,\varphi_\theta(t,\xi,\mu),(\id, \varphi_\theta (t,\cdot, \mu))_\sharp \mu)]=0.
\end{equation}
\item 
 For all 
 $(t,\mu)\in [0,T]\times \cP_2(\sR^n)$,
there exists a neighborhood
$\cO_{t,\mu}(\theta)\subset \sR^k$
of $\theta$ such that for all $\theta'\in \cO_{t,\mu}(\theta) $,
the following process
\begin{align}
\label{eq:martingale}
    \begin{split}
     &\bigg( e^{-\beta s}\hat{V}_D(s,X_s^{t,\xi,\theta'},\sP_s^{t,\mu,\theta'})
        \\
        &\quad +\int_t^se^{-\beta u}(r-\hat{q}_D)(u,X_u^{t,\xi,\theta'},\varphi_{\theta'}(u,X_u^{t,\xi,\theta'},\sP_u^{t,\mu,\theta'}),
        (\id, \varphi_{\theta'} (u,\cdot, \sP_u^{t,\mu,\theta'}))_\sharp \sP_u^{t,\mu,\theta'})\d u\bigg)_{s\in [t,T]}
    \end{split}
\end{align}    
is an $\sF$-martingale, 
where $X^{t,\xi,\theta'}$
and $\sP^{t,\mu,\theta'}$
are defined as in Theorem \ref{thm:dpg_flow}.

\end{enumerate}
Then the functions $\hat V$ and $\hat q$ defined by 
\eqref{eq:decompose_V} and 
\eqref{eq:decompose_q}, respectively, satisfy
for all $(t,\mu,\theta')\in [0,T]\times \cP_2(\sR^n)\times  \cO_{t,\mu}(\theta)$,
\begin{equation} 
\label{eq:characterization_integrated}
\hat V(t,\mu)=V^\varphi(t,\mu,\theta),\quad  
\hat q(t,\mu,\theta') =A^\varphi[V^\varphi(\cdot,\theta)](t,\mu,\theta').
\end{equation}
 Moreover, 
\begin{align}\label{eq:policy_grad}
 \begin{split}
\nabla_\theta J(\varphi_\theta) 
&
  =  
 \sE \bigg[   \int_0^T
   e^{-\beta s} 
\partial_\theta\varphi_\theta(s,X_s^{\theta},\sP_s^{ \theta})^\top 
 \bigg\{   (\partial_a 
   \hat q_D)   
    (s, X_s^{ \theta},
\varphi_\theta(s,X_s^{ \theta},\sP_s^{ \theta}),
 (\id, \varphi_{\theta} (s,\cdot, \sP_s^{ \theta}))_\sharp \sP_s^{ \theta}
) 
\\
&\quad +
\widetilde{\sE}\left[(\partial_\nu 
   \hat q_D)   
    \left(s, \widetilde{X}_s^{ \theta}, 
\varphi_\theta(s,\widetilde{X}_s^{\theta},\sP_s^{\theta}),
 (\id, \varphi_{\theta} (s,\cdot, \sP_s^{\theta}))_\sharp \sP_s^{\theta}
\right)\right] ({X}_s^{\theta},
\varphi_{\theta} (s,{X}_s^{ \theta}, \sP_s^{\theta})
)
\bigg\}    
   \d s
   \bigg], 
 \end{split}
 \end{align}
where  $({X}_s^{\theta},\sP_s^{\theta})=(\widetilde{X}_s^{0,\sP_{\xi_0},\theta},
\sP_s^{0,\sP_{\xi_0},\theta})$, 
$\widetilde{X}_s^{\theta}$ is an independent copy of ${X}_s^{\theta}$,
and
$\partial_\nu \hat q_D(\cdot, \Gamma)$ is the partial  L-derivative of $\hat q_D$ with
respect to the second marginal of $\Gamma\in \cP_2(\sR^n\times \sR^d)$.
 
\end{thm}

{Analytically, 
{Theorem \ref{thm:martingale} is the continuous-time-space analogue of the 
discrete-time MFC with learning \cite{gu2023dynamic} where the  lifted Q-function (called IQ-function) 
defined on the space of policies 
is represented as a local Q-function and a  local policy  $h: \mathcal{S}\rightarrow \mathcal{P}(\mathcal{A})$.
These local value functions are analogous to the decomposed value functions studied in the MFC  literature (see, e.g., \cite{buckdahn2017mean}). 
This decomposition is also consistent with the idea of ``centralized training with decentralized execution" \citep{foerster2016learning, lowe2017multi}, which has been adopted to improve the computational efficiency of RL algorithms for MFC \citep{gu2025mean, bayraktar2025finite}.
In particular, it enables both the value function and the reward to be decomposed additively across individual observations. 
Specifically,
Theorem~\ref{thm:martingale} expresses  
the functions $ \hat V$  and $\hat q$  as integrals of a local value function $\hat V_D$ and a local advantage rate function $\hat q_D$, respectively, as given in \eqref{eq:decompose_V} and \eqref{eq:decompose_q}. 
Such local representations allow 
learning $\hat V_D$ and $\hat q_D$  from observed state and control trajectories, yielding a more informative learning signal than directly learning 
$\hat V$ and $\hat q$ from the flow of state laws as in \eqref{eq:constant_flow_policy}; see   Section    \ref{sec:alg}.
 
Note, however,
 only the integrated value and advantage functions  $\hat V$ and $\hat q$,  rather than their local counterparts, can be uniquely determined, as shown in \eqref{eq:characterization_integrated}. Importantly, these integrated quantities are exactly those needed for policy updates.
 
}}

\subsection{Existence  of $\hat V_D$ and $  \hat q_D$ in Theorem~\ref{thm:martingale}}

The function   $\hat V_D$ 
can be taken as the (decomposed) value function for a control problem with decoupled state dynamics studied as in \cite{buckdahn2017mean,frikha2025actor}, while the function 
$\hat q_D$ can be taken as the \emph{integrated Hamiltonian over the joint state–action distribution}.
To see it,
assume without loss of generality   that $\beta=0$. 
For all $(t,x,\mu)\in [0,T]\times \sR^n\times \cP_2(\sR^n)$,
 let $ \xi\in  L^2(\mathcal{F}_t; \sR^n)$ with law $\mu$, and define the following  value function associated with  the policy $\varphi_\theta$:
\begin{align}
\label{eq:decoupled_cost_mv_dpg}
V_D^\theta(t,x,\mu) \coloneqq \sE \left[
\int_t^T 
r^\varphi (s, X^{t,x, \mu, \theta}_s, \sP_{X^{t,\xi, \theta}_s},\theta)  \d s +
g(X^{t,x,\mu, \theta}_T, \sP_{X^{t,\xi, \theta}_T})
\right],
\end{align}
where
$X^{t, \xi,\theta}$
and 
$X^{t, x, \mu,\theta}$ satisfy the following decoupled state dynamics:
for all $s\in [t,T]$,
\begin{align}
\label{eq:decoupled_dynamics}
\begin{split}
\d X^{t, \xi,\theta}_s &=b^\varphi(s,X^{t, \xi,\theta}_s, \sP_{X^{t, \xi,\theta}_s}, \theta)\d s+\sigma^\varphi(s,X^{t, \xi,\theta}_s, \sP_{X^{t, \xi,\theta}_s},\theta)\d W_s,
  \quad X^{t, \xi,\theta}_t=\xi,
 \\
 \d X^{t, x, \mu,\theta}_s &=b^\varphi(s,X^{t, x, \mu,\theta}_s, \sP_{X^{t, \xi,\theta}_s},\theta)\d s+\sigma^\varphi(s,X^{t, x, \mu,\theta}_s, \sP_{X^{t, \xi,\theta}_s},\theta)\d W_s,
  \quad X^{t, x, \mu,\theta}_t=x.
 \end{split} 
\end{align}
Note that by the weak uniqueness of \eqref{eq:decoupled_dynamics}, 
the measures $(\sP_{X^{t, \xi,\theta}_s})_{s\in [t,T]}$ depend on $\xi$ only through its law $\mu$, which allows us to   regard   $X^{t, x, \mu,\theta}$ as a function of   $\mu$  without specifying the choice of   $\xi$.
Suppose that $\hat V_D$ is  sufficiently regular, by It\^o's formula \cite[Proposition 5.102]{carmona2018probabilistic},
\begin{align}
\label{eq:martingale_decoupled}
[t,T]\ni s\mapsto  M_s\coloneqq {V}^\theta_D(s,X_s^{t,\xi,\theta'},\sP_s^{t,\mu,\theta'}) 
 &  -\int_t^s 
  \mathcal L^\varphi_D[{V}^\theta_D] ( r,X_r^{t,\xi,\theta'},\sP_r^{t,\mu,\theta'},\theta') \d r
\end{align}
is  an $\sF$-martingale, where  
$\mathcal L^\varphi_D$ is the generator of \eqref{eq:sde_mv_perturbed} satisfying
 \begin{align}
 \label{eq:generator_decoupled}
 \begin{split}
 &\mathcal L^\varphi_D[ w](t,x,  \mu,\theta)
 \\
 &\coloneqq
 \partial_t w(t,x,\mu)
      +  b^\varphi(t,x,\mu,\theta)\cdot \partial_x w(t,x,\mu) +\frac{1}{2}   \sigma^\varphi[\sigma^\varphi]^\top  (t,x,\mu,\theta)\colon     
     \partial^2_x w(t,x,\mu)
     \\
    &\quad + \sE_{\xi\sim \mu}\left[  b^\varphi(t,\xi,\mu,\theta)\cdot \partial_\mu w(t,x,\mu) (
    \xi) +\frac{1}{2}  [\sigma^\varphi(\sigma^\varphi)^\top] (t,\xi,\mu,\theta)\colon     
     \partial_v \partial_\mu w(t,x,\mu)(\xi) \right].
     \end{split}
\end{align}
Define for all $(t,x,a,\Gamma)\in [0,T]\times\sR^n\times \sR^d\times \cP_2(\sR^n\times \sR^d)$,  $\mu(dx)\coloneqq \Gamma(\d x, \sR^d)$, and   
\begin{equation}
\label{eq:q_verification}
    \begin{aligned}
        & q_D(t,x,a, \Gamma)\\
       & \coloneqq \partial_t V_D^\theta(t,x,\mu)
      +  b(t,x,a, \Gamma)\cdot \partial_x V_D^\theta(t,x,\mu) +\frac{1}{2}   [\sigma \sigma^\top]  (t,x,a, \Gamma)\colon     
     \partial^2_x V_D^\theta(t,x,\mu) \\
     &\quad + \sE_{(\xi,\alpha)\sim \Gamma}\left[  b(t,\xi,\alpha,\Gamma)\cdot \partial_\mu V_D^\theta(t,x,\mu) (
    \xi) +\frac{1}{2}  [\sigma  \sigma^\top] (t,\xi,\alpha,\Gamma)\colon     
     \partial_v \partial_\mu V_D^\theta(t,x,\mu)(\xi) \right]
     \\
     &\quad +r(t,x,a,\Gamma).
    \end{aligned}
\end{equation}
Since 
$
    h^\varphi (t, x, \mu,\theta) = 
    h(t,x,\varphi_{\theta}(t,x,\mu),(\id,\varphi_{\theta}(t,\cdot,\mu))_\sharp\mu)
$ for all $h\in \{b,\sigma, r\}$,
  for all $(t,x,\mu)\in [0,T]\times \sR^n\times \cP_2(\sR^n)$ and $\theta'\in \sR^k$,
 \begin{equation*} 
    \begin{aligned}
        &\mathcal{L}_D^\varphi[V_D^\theta](t,x,\mu,\theta') = [q_D-r](t,x,\varphi_{\theta'}(t,x,\mu),(\id,\varphi_{\theta'}(t,\cdot,\mu))_\sharp\mu),
    \end{aligned}
\end{equation*}
which along with \eqref{eq:martingale_decoupled} implies that $q_D$ satisfies the martingale condition \eqref{eq:martingale}.
Applying It\^o's formula to $s\mapsto \sE[{V}^\theta_D(s,X_s^{t,\xi,\theta},\sP_s^{t,\mu,\theta})]$  
further shows that  $V_D$ and $q_D$ satisfy \eqref{eq:bellman},
which proves that  $V_D$ and $q_D$ are candidate functions satisfying Theorem \ref{thm:martingale}.

We emphasize that in \eqref{eq:q_verification}, it is essential to consider a  function  $q_D$ on a lifted space  $\cP_2(\sR^n\times \sR^d)$,
even when all coefficients are independent of the control distribution  as in classical MFC problems.
 This is because \eqref{eq:generator_decoupled}
requires simultaneously integrating the coefficients $b$, $\sigma$, and  the policy   $\varphi_\theta$
over the state component. The characterization of the advantage function thus  necessarily involves integrating with respect to the joint state–action law induced by the   policy $\varphi_\theta$.

\section{Model-free Advantage Actor-Critic Algorithm}\label{sec:alg}

Building on Theorem~\ref{thm:martingale}, we now develop a model-free   advantage actor–critic RL algorithm for extended MFC problems based on   neural network (NN) parameterizations.
We first outline the key components of the framework and then provide implementation details. In the sequel, we denote by  $V_\phi,q_\psi$ and $\varphi_\theta$    generic NN  approximations  of the value function, advantage rate function, and policy, respectively. 

\subsection{Algorithmic design}

\paragraph{Neural network architecture with measure  inputs.}

Although both Theorem~\ref{thm:dpg_flow} and Theorem~\ref{thm:martingale} characterize the DPG for a given policy, Theorem~\ref{thm:martingale} provides a more effective representation of the associated value and advantage rate functions than    Theorem~\ref{thm:dpg_flow}, thereby facilitating more efficient learning.

Specifically, Theorem~\ref{thm:dpg_flow} represents the value function $\hat V$ as a function of time and the state measure,
and 
the advantage rate function
$\hat q$ as a function of time, the state measure, and the policy parameter.
Such  measure-dependent functions can be approximated by cylinder functions based on finite-dimensional features of the distribution, namely
\begin{align}
\label{eq:V_q_NN}
\hat{V}(t,\mu)\approx NN_\phi(t,\E_{\xi\sim\mu}[\Psi(\xi)]),
\quad 
\hat{q}(t,\mu,\theta)
\approx NN_\psi(t,\E_{\xi\sim\mu}[\widetilde{\Psi}(\xi)],\theta),
\end{align}
where $NN_\phi$ and 
$NN_\psi$
denote  generic NNs, and $\Psi, \tilde{\Psi}:\R^n\to\R^m$ 
are some   prescribed or learnable feature maps  
\cite{germain2022deepsets, mekkaoui2026learning,soner2025learning}. 
However, updating the value and advantage rate functions according to \eqref{eq:constant_flow_policy} requires sampling multiple state trajectories to estimate the state distribution, which is not sample efficient.
Moreover, 
since NN policies typically involve high-dimensional parameter spaces and complex architectures, it is often challenging to design and learn $\hat q $
 as a function that directly maps policy parameters to the value of the advantage rate function.

In contrast, Theorem \ref{thm:martingale} introduces a stronger martingale condition to  identify the   value function $\hat V$
and    advantage rate function $\hat q$ through the local value function $\hat V_D$, and the local advantage rate function $\hat q_D$.
This motivates parameterizing  $\hat V_D$ and $\hat q_D$ directly  as follows:
\begin{align}
\label{eq:V_q_D_NN}
\hat{V}_D(t,x,\mu)\approx NN_\phi(t,x,\E_{\xi\sim\mu}[\Psi(\xi)]),
\quad 
\hat{q}_D(t,x,a,\Gamma)
\approx NN_\psi(t,x,a,\E_{(\xi,\alpha)\sim\Gamma}[\widetilde{\Psi}(\xi,\alpha)]).
\end{align}
These neural networks can be viewed as feature extractors, enabling accurate approximations of   $\hat V$ and $\hat q$ even with simple feature maps $\Psi$ and 
$\widetilde{\Psi}$, such as polynomials. 
Moreover, the parameterization of $\hat q_D$ in \eqref{eq:V_q_D_NN}  
takes action as an input variable rather than the policy parameter as in \eqref{eq:V_q_NN}.
Since the action space is typically much lower-dimensional than the policy parameter space, and more directly related to the state dynamics, parameterizing the map from action to the value of the advantage rate function leads to a simpler approximation problem and improves learning efficiency.
Finally, 
the martingale criterion \eqref{eq:martingale} also yields more informative temporal-difference updates of the NNs in \eqref{eq:V_q_D_NN} directly based  on observed state and control trajectories, thereby improving learning efficiency, as shown in Section~\ref{sec:experiment}.

Based on the above observations,  
we adopt 
\eqref{eq:V_q_D_NN}
to parameterize  
$\hat{V}_D$ and $\hat{q}_D$, and 
denote  them by $V_\phi $ and $ q_\psi $, respectively.

\paragraph{Learning critics.}
We   train $V_\phi $ and $ q_\psi $ 
so that they satisfy the conditions   \eqref{eq:bellman} and \eqref{eq:martingale}. Following similar arguments as   in \cite{jia2023q, wei2025continuous, cheng2025deterministic, guo2026deterministic}, the martingale condition \eqref{eq:martingale} can be enforced by updating      $V_\phi$ and $q_\psi$ through the following temporal-difference learning scheme:
\begin{equation}\label{eq:td_v}
    \phi\leftarrow\phi-\eta\E\left[\partial_\phi V_\phi(t,x_{t},\mu_t) \Big(V_\phi(t,x_{t},\mu_t)-[r-q_\psi](t,x_{t},a_t,\Gamma_t)h-e^{-\beta h}V_\phi(t+h,x_{t+h},\mu_{t+h})\Big)\right],
\end{equation}
\begin{equation}\label{eq:td_q}
    \psi\leftarrow\psi-\eta\E\left[\partial_\psi q_\psi(t,x_{t},a_t,\Gamma_t) \Big(V_\phi(t,x_{t},\mu_t)-[r-q_\psi](t,x_{t},a_t,\Gamma_t)h-e^{-\beta h}V_\phi(t+h,x_{t+h},\mu_{t+h})\Big)\right],
\end{equation}
where $h>0$ is the   time step size, and  
$(x_t,\mu_t)$ represents the state process and its distribution at time $t$. 
To meet the Bellman constraint \eqref{eq:bellman}, 
we re-parameterize the advantage rate function 
\begin{equation}\label{eq:repara_q}
    q_\psi\big(t,x,a,\Gamma\big)\coloneqq\bar{q}_\psi\big(t,x,a,\Gamma\big)-\E_{\xi\sim\mu}\bar{q}_\psi\Big(t,\xi,\varphi_\theta(t,\xi,\mu),(\id,\varphi_\theta(t,\cdot,\mu))_\sharp\mu)\Big),
\end{equation} 
where $\bar{q}$ is an NN of the form \eqref{eq:V_q_D_NN},
and $\varphi_\theta$ represents the current deterministic policy.  
To enforce 
  the terminal condition in \eqref{eq:bellman},
we introduce a penalty term
\begin{equation}
    \E (V_\phi(T,x_T,\mu_T)-g(x_T,\mu_T))^2,
\end{equation}
and minimize it with respect to $\phi$ via stochastic gradient descent based on the observed $ x_T,\mu_T$ and $g(x_T,\mu_T)$.

\paragraph{Exploration.}

The use of a deterministic actor does not eliminate exploration. During training, exploration can be introduced either by perturbing the actions generated by the current policy (see e.g., \cite{szpruch2024optimal, cheng2025deterministic}) or by perturbing the parameters of the policy network \citep{plappert2017parameter}. The essential distinction is that randomness is used as a data-collection mechanism rather than being encoded into the policy class itself. The learned policy remains a deterministic feedback control, while exploration can be adjusted according to the geometry of the action and parameter spaces. 


In our implementation, 
 given $\sigma_{\text{epl}}>0$,
we consider the following two exploration mechanisms: 
\begin{itemize}
   
    \item Approach I (action space exploration): Perform $a_t\sim \mathcal{N}(\varphi_\theta(t,x_t,\mu_t),\sigma_{\text{epl}}^2I)$.
     \item Approach II (parameter space exploration): Perform $a_t= \varphi_{\theta'}(t,x_t,\mu_t)$ where $\theta'\sim\mathcal{N}(\theta,\sigma_{\text{epl}}^2I)$. 
\end{itemize}

Approach I explores the action space by perturbing deterministic policies with noise, which is the standard approach in single-agent RL (see e.g., \cite{szpruch2024optimal, cheng2025deterministic}).
 In general, this approach may not  fully explore the parameter space, especially when  the Jacobian   $\partial_\theta 
 \varphi_\theta$  is  non-invertible. 
 
Approach II explores the parameter space and is more consistent with Theorem~\ref{thm:martingale}. However, as shown in \citep{plappert2017parameter}, its performance is more  sensitive to the choice of noise scale $\sigma_{\text{epl}}$ than Approach I. This is because the effect of perturbing the parameters on the induced actions depends on the geometry of the policy parameterization, which is locally governed by the Jacobian of the policy network. Consequently, depending on the network architecture, parameter-space noise can lead to unstable and unpredictable behavior in the action space, necessitating careful tuning of the exploration noise scale.

 Our experiments  suggest that, with appropriate tuning of the perturbation scale in Approach II, the two approaches achieve comparable performance in standard mean field control problems. However, in certain extended MFC problems, the dependence on the control law can make Approach I less effective, resulting in slower convergence.
  See Section \ref{subsec:liquidation} for details.


\subsection{Implementation details}

\paragraph{McKean-Vlasov dynamics simulator.}

We construct a black-box simulator of the McKean–Vlasov dynamics \eqref{eq:sde_mv_cl} using a standard particle approximation with Euler–Maruyama discretization \cite{reisinger2024fast}, and observe the resulting state and control trajectories as well as the associated rewards.

Specifically, 
let $h>0$ be the time step size, 
and $M\in \mathbb N$
be the number of  homogeneous agents (particles). 
Let $(x_{kh,j},a_{kh,j})$   be the state-action of the $j$-th agent at time $kh$, for $j=1,\ldots, M$, define the empirical measures   
$ 
    \hat \mu_{kh}\coloneqq \frac{1}{M}\sum_{j=1}^M\delta_{x_{kh,j}}$ and  $\hat 
\Gamma_{kh}\coloneqq\frac{1}{M}\sum_{j=1}^M\delta_{(x_{kh,j},a_{kh,j})}$, and simulate the   states at time $(k+1)h$ by 
\begin{equation}
    x_{(k+1)h,j}\coloneqq x_{kh,j}+  b(kh,x_{kh,j},a_{kh,j},\hat \Gamma_{kh})h+\sigma(kh,x_{kh,j},a_{kh,j},\hat \Gamma_{kh})\sqrt{h} Z_{kh,j},\quad 1\leq j\leq M,
\end{equation}
where $\{Z_{kh,j} \}_{ k=0,\ldots, \lceil T/h\rceil, j=1,\ldots, M} $ are independent $m$-dimensional standard normal random vectors. The observed  instantaneous reward  at time $kh$ is given by $r_{kh,j}=r(kh,x_{kh,j},a_{kh,j},\hat \Gamma_{kh})$,
and the reward at time $T$ is given by
$r_{T,j}=g(x_{T,j},\hat \mu_{T})$, for all $j=1,\ldots, M$.

\paragraph{Critic update.}
Due to the off-policy nature of Algorithm \ref{alg:ddpg}, 
we store 
the observed transition $\big(kh,x_{kh,j},a_{kh,j},r_{kh,j}, x_{(k+1)h,j}\big)_{j=1}^M$  in a replay buffer $\mathcal{R}$, and re-sample them later to update actor and critic. 
We also  store the terminal state  and reward  $(Kh,x_{Kh,j},r_{Kh,j})_{j=1}^M$ in $ \mathcal{R}$.
For brevity, we denote by   
$\tilde{x}_t$ the concatenation of time and state $(t,x_t)$. 
 We also employ a target value network $V_{\phi^{tgt}}$, 
whose parameters are updated as an exponential moving average of the value network parameters. This technique is widely used in modern deep RL algorithms, including DDPG \citep{lillicrap2015continuous} and SAC \citep{haarnoja2018soft}, to improve training stability.

During each episode, after every $m\geq 1$ steps of simulation, we sample a batch of transitions $ \{\mathcal{D}^{(i)}\}_{i=1}^B$ from replay buffer $\mathcal{R}$ 
where  
$\mathcal{D}^{(i)}=\big(\tilde{x}^{(i)}_{k_ih,j}, a^{(i)}_{k_ih,j}, r^{(i)}_{k_ih,j},\tilde{x}^{(i)}_{(k_i+1)h,j}\big)_{j=1}^M$,
corresponding to different trajectories and time indices, 
and define   the martingale loss by
\begin{align}\label{eq:martingale_loss}
    \begin{split}
    &\mathcal{L}^{\mathcal{M}} \coloneqq\frac{1}{BM}\sum_{i=1}^B\sum_{j=1}^M \Big(V_\phi(\tilde{x}^{(i)}_{k_ih,j},\hat \mu^{(i)}_{k_ih})-[r^{(i)}_{k_ih,j}-q_\psi(\tilde{x}^{(i)}_{k_ih,j},a^{(i)}_{k_ih,j},\hat \Gamma^{(i)}_{k_ih})]h-e^{-\beta h}V_{\phi^{tgt}}(\tilde{x}^{(i)}_{(k_i+1)h,j},\hat  \mu^{(i)}_{(k_i+1)h})\Big)^2,
    \end{split}
\end{align}
where $q_\psi$ is given by  \eqref{eq:repara_q}.
Note that $\partial_\phi\mathcal{L}^\mathcal{M}$ and $\partial_\psi\mathcal{L}^\mathcal{M}$ are stochastic estimates of the semi-gradients in \eqref{eq:td_v} and \eqref{eq:td_q}, respectively.
To enforce the terminal constraints, we also sample a batch of terminal states $\{\mathcal{D}_{ter}^{(i)}\}_{i=1}^B$ from $\mathcal{R}$   from different trajectories
and compute the terminal loss
\begin{equation}\label{eq:terminal_loss}
    \mathcal{L}^\mathcal{T}=\frac{1}{BM}\sum_{i=1}^B\sum_{j=1}^M(V_\phi(\tilde{x}^{(i)}_{Kh,j},\hat \mu^{(i)}_{Kh})-r^{(i)}_{Kh,j})^2.
\end{equation}
The   losses  $\mathcal{L}^\mathcal{M}$ and  $\mathcal{L}^\mathcal{T}$ are combined to update the critics $V_\phi$ and $q_\psi$
via gradient descent.

\paragraph{Actor update.}
To implement the exact policy gradient of the current policy $\varphi_\theta$ given in \eqref{eq:policy_grad},  we sample a batch $\{\mathcal{D}^{(i)}\}_{i=1}^B$ from $\mathcal{R}$,
and define the  actor objective by
\begin{equation}\label{eq:policy_loss}
    \mathcal{L}^\mathcal{A}=\frac{1}{MB}\sum_{j=1}^M\sum_{i=1}^B q_\psi\big(\tilde{x}_{k_ih,j},\varphi_\theta(\tilde{x}_{k_ih,j},\hat \mu_{k_ih}),\Gamma^\theta_{k_ih}\big)h,\quad \Gamma^\theta_{k_ih}\coloneqq\frac{1}{M}\sum_{j=1}^M\delta_{(x_{k_ih,j},\varphi_\theta(\tilde{x}_{k_ih,j},\mu_{k_ih}))}. 
\end{equation}
Note that $\partial_\theta\mathcal{L}^\mathcal{A}$ is a stochastic estimator of policy gradient \eqref{eq:policy_grad}.

\paragraph{Full algorithm.}

\begin{algorithm}[!htbp]
    \caption{\textbf{C}ontinuous \textbf{T}ime \textbf{D}eep \textbf{D}eterministic \textbf{P}olicy \textbf{G}radient}
    \label{alg:ddpg}
    \begin{algorithmic}
        \STATE{{\bfseries Inputs:} Discretization stepsize $h$, horizon $K=T/h$, number of episodes $N$, number of agents $M$, policy net $\varphi_\theta$, advantage-rate net $\bar{q}_\psi$, value net $V_\phi$, update frequency $m$, exploration noise $\sigma_{\text{epl}}$, soft update parameter $\tau$, learning rate $\eta$, batch size $B$, terminal   constraint weight $w$}
            \STATE{{\bfseries Initialization:}
                $\phi,\psi,\theta$, target $\phi^{tgt}=\phi$, and replay buffer $\mathcal{R}=\emptyset$
            }
            \FOR{$n = 1, \cdots, N$}
                \STATE{Observe      states $\big(\tilde{x}_{0,j}\big)_{j=1}^M$}
                \FOR{$k=0,\cdots,K-1$}
                    \STATE{$\triangleright$\colorbox{Ocean}{\emph{Take action}}}
                    \STATE{
                        \textbf{Option I.} Perform $a_{kh,j}\sim \mathcal{N}(\varphi_\theta(\tilde{x}_{kh,j},\hat \mu_{kh}),\sigma_{\text{epl}}^2I)$ for each $j\in[M]$
                    }
                    \STATE{
                        \textbf{Option II.} Perform $a_{kh,j}= \varphi_{\theta'}(\tilde{x}_{kh,j},\hat \mu_{kh})$ where $\theta'\sim\mathcal{N}(\theta,\sigma_{\text{epl}}^2I)$ for each $j\in[M]$
                    }
                    \STATE{
                         Observe $\mathcal{D}=\big(\tilde{x}_{kh,j},a_{kh,j},r_{kh,j}, \tilde{x}_{(k+1)h,j}\big)_{j=1}^M$ and store them in $\mathcal{R}$
                    }
                    \IF{$k \equiv 0 \textbf{ mod } m$}
                    \STATE{$\triangleright$\colorbox{Ocean}{\emph{Update critics}}}
                    
                    \STATE{
                        Sample $\{\mathcal{D}^{(i)}\}_{i=1}^B$ from $\mathcal{R}$
                        and compute $\mathcal{L}^\mathcal{M}$ as in \eqref{eq:martingale_loss}
                    }
                    
                    \STATE{
                        Sample terminal states $\{\mathcal{D}_{ter}^{(i)}\}_{i=1}^B$ from $\mathcal{R}$
                        and compute $\mathcal{L}^\mathcal{T} $ as in \eqref{eq:terminal_loss}
                    }
                    \STATE{
                        Update the critics: $\psi\leftarrow\psi-\eta\partial_\psi \mathcal{L}^M$,
                        $\phi\leftarrow\phi-\eta\partial_\phi (\mathcal{L}^M+w\mathcal{L}^\mathcal{T})$
                    }
                      \STATE{
                        Update the target: 
                        $\phi^{tgt}\leftarrow \tau \phi+(1-\tau)\phi^{tgt}$
                    }\STATE{$\triangleright$\colorbox{Ocean}{\emph{Update actor}}}
                    
                    \STATE{
                        Sample $\{\mathcal{D}^{(i)}\}_{i=1}^B$ from $\mathcal{R}$
                        and compute the policy loss $\mathcal{L}^\mathcal{A} $ as in \eqref{eq:policy_loss}
                    }
                    \STATE{
                        Update the actor: $\theta\leftarrow \theta+\eta \partial_\theta\mathcal{L}^\mathcal{A}$
                    }
                  
                    \ENDIF
                \ENDFOR
            \ENDFOR
    \end{algorithmic}
\end{algorithm}

The full procedure of  \textbf{C}ontinuous \textbf{T}ime \textbf{D}eep \textbf{D}eterministic \textbf{P}olicy \textbf{G}radient (CT-DDPG) for extended MFC problems is summarized in Algorithm \ref{alg:ddpg}.

\section{Numerical Experiments}

\label{sec:experiment}

In this section, we illustrate the efficiency of the proposed CT-DDPG algorithm within the general learning MFC framework. Unless otherwise specified, CT-DDPG refers to the implementation with action space exploration (Option I in Algorithm \ref{alg:ddpg}).

\subsection{Consensus control of Cucker-Smale model}
We consider the consensus control of multi-dimensional stochastic Cucker-Smale (C-S) models studied in \citep{nourian2011mean,reisinger2025convergence}. 
In this problem, 
  the agent  aims to enforce consensus
emergence of an interactive particle system via   external intervention. 

Let $T>0$, and consider the  following $2n$-dimensional controlled McKean-Vlasov dynamics, which can be viewed as the large
population limit of the finite-particle model studied in \cite{cucker2007emergent}:
for all $t\in [0,T]$,  
\begin{equation}
\label{eq:cs_model}
    \d z_t=v_t\d t,\qquad \d v_t=\left(\alpha_t+\int_{\sR^n\times \sR^n} \kappa(z_t,v_t,z',v')\P_{(z_t,v_t)}(\d z',\d v')\right)\d t+\sigma\d W_t,
\end{equation}
with an initial state $(z_0,v_0)\in L^2(\cF_0;\sR^n\times \sR^n)$, 
where $\sigma\in\R^{n\times n}$, $W$ is an $n$-dimensional Brownian motion,   and the interaction kernel $\kappa:\R^n\times\R^n\times\R^n\times \R^n\to \R^n$ is defined by
\begin{equation}
    \kappa(z,v,z',v')=\frac{C(v'-v)}{(1+\|z'-z\|^2)^\gamma},\quad \ \text{with some }C,\gamma\geq 0.
\end{equation}
The variables $z$ and $v$ represent the position and velocity of a representative particle, respectively. 
The dynamics \eqref{eq:cs_model} models the self-organization behavior of a crowd, where the interaction kernel $\kappa$ captures the decay of influence between agents as their distance increases. It is known that, in the absence of external control, flocking behavior, namely, the convergence of the velocity trajectories $v$ to a common limit as $t\to \infty$, occurs only when the parameter $\gamma$ is sufficiently small.

The aim of the agent  is to either induce consensus on models that would otherwise diverge, or to
accelerate the  self-organization behavior. Specifically, given $c>0$, 
we consider  maximizing the following   reward functional over all adapted controls $\alpha$:
\begin{equation}
    J(\alpha) = -\E\left[\int_0^T ( |v_t-\E[v_t] |^2+c |\alpha_t |^2)\d t +  |v_T-\E[v_T] |^2\right] 
\end{equation}
Note that in the special case with  $\gamma=0$, it is a linear-quadratic (LQ)  mean field   control problem.
In this case, the    optimal policy is linear and given  explicitly   by  (see e.g.,  \citep{yong2017linear}):
\begin{equation}
\label{eq:CS_LQ}
    \varphi^*(t,z,v,\mu)=-\frac{P(t)}{c}(v-\E_{v'\sim\mu}[v']),
\end{equation}
where $\frac{\d}{\d t}P(t)-\frac{1}{c}P(t)^2-2CP(t)+1=0$ with $P(T)=1$.

\paragraph{Baselines.}
We compare our \textbf{CT-DDPG} with existing mean-field RL algorithms, including \textbf{Actor-Critic (AC)} \citep{frikha2025actor}, which requires the action dependence of dynamics to be explicitly known, and \textbf{q-Learning} \citep{wei2025continuous}, which relies on a closed-form policy given by state-dependent Gibbs measures.
Due to the model-aware nature of these two methods, and especially their incompatibility with deep RL framework, we evaluate them only in the LQ case.
In addition to the default procedures described in Algorithm \ref{alg:ddpg}, we also consider a variant \textbf{CT-DDPG-Flow}, which learns the value function $\hat{V}$ directly based on the deterministic flow characterization in Theorem \ref{thm:dpg_flow}, while still using the decoupled advantage rate estimator $\hat{q}_D$ to represent $\hat q$.

\paragraph{Model architecture.}
Across all experiments, the policy, advantage network, and value network in CT-DDPG and its variants are implemented as three-layer fully connected MLPs with ReLU activations and hidden width $400$. To incorporate temporal information, we augment the environment observations with a sinusoidal embedding, yielding $\tilde{x}_t=(\cos(\tfrac{2\pi t}{T}),\sin(\tfrac{2\pi t}{T}),z_t,v_t)$, where $T$ denotes the maximum horizon. To embed the mean-field distribution, the default parameterization of CT-DDPG is given by
\begin{equation}\label{eq:value_param}
    V_\phi(\tilde{x}_t,\mu_t)=NN_{\phi_1}\left(\tilde{x}_t,\E_{x'\sim\mu_t}NN_{\phi_2}(x')\right),  \tag{CT-DDPG}
\end{equation}
where $\phi=(\phi_1,\phi_2)$ and similarly for $\varphi_\theta$ and $\bar{q}_\psi$.
Here we use a learnable feature map to preserve the model-agnostic nature of our method.
We also consider a variant based on prescribed moment features of the distribution, referred to as \textbf{CT-DDPG-Moment}:
\begin{equation}\label{eq:value_param_moment}
    V_\phi(\tilde{x}_t,\mu_t)=NN_\phi\left(\tilde{x}_t,\E_{x'\sim\mu_t}[x'],\Var_{x'\sim\mu_t}[x']\right),  \tag{CT-DDPG-Moment}
\end{equation}
As a benchmark, 
we further evaluate a variant that incorporates prior knowledge of the dynamics through a hand-crafted feature map, which we denote by \textbf{CT-DDPG-Kernel}:
\begin{equation}\label{eq:value_param_kernel}
    V_\phi(\tilde{x}_t,\mu_t)=NN_\phi\left(\tilde{x}_t,\E_{(z',v')\sim\mu_t}\left[\frac{v'-v_t}{(1+\|z'-z_t\|^2)^\gamma}\right]\right),  \tag{CT-DDPG-Kernel}
\end{equation}
This variant requires the exponent $\gamma$ to be specified a priori, and is therefore less model-agnostic than the default parameterization. Nevertheless, all these variants remain model-free in the sense that they do not exploit closed-form expressions for the optimal policy or value function, in contrast to previous works such as \cite{frikha2025actor,wei2025continuous}.

For AC and q-Learning, we follow \cite{frikha2025actor,wei2025continuous} respectively and adopt the model-aware LQ parameterization. Specifically, the value function is parameterized as a quadratic function of the state and the policy as a linear function of the state with time dependence incorporated through a three-layer MLP based embedding.

\paragraph{Training hyperparameters.}
We set dimension $n=3$ and simulate $M=50$ homogeneous agents to approximate the McKean-Vlasov dynamics. To accelerate training, we run 8 environments in parallel, i.e., collecting 8 trajectories per episode. All neural networks are optimized by Adam with a learning rate of $3\times 10^{-4}$ and a batch size of $B=256$. The update frequency is $m=1$. The soft target update parameter is $\tau=0.1$. The weight for the terminal value constraint is $w=0.002$. And the default exploration noise scale is $\sigma_{\text{epl}}=0.1$.
The entropy regularization coefficient in AC and q-Learning is $0.01$.

\begin{figure}[!ht]
    \centering
    \includegraphics[width=0.9\linewidth]{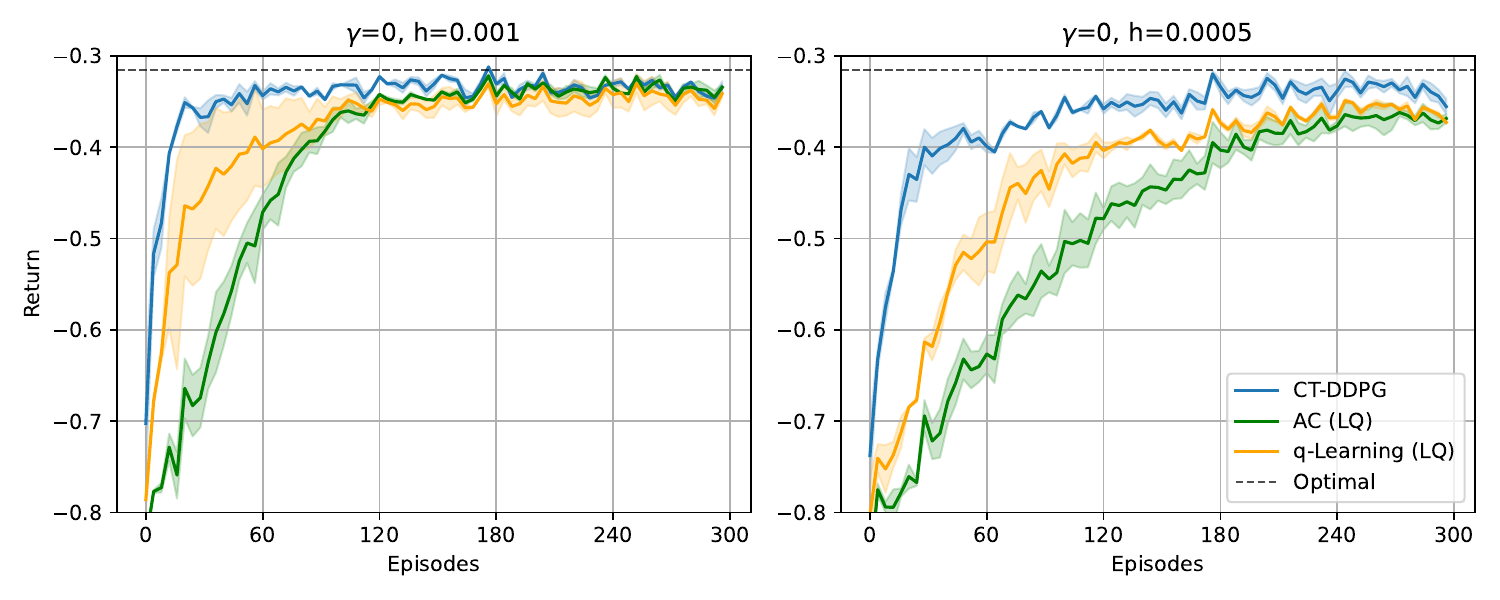}
    \caption{Results for C-S model with $\gamma=0$. AC and q-Learning exploit the LQ structure and corresponding parameterization, while CT-DDPG uses neural network parameterized policy.}
    \label{fig:cs_lq}
\end{figure}

\begin{figure}[!ht]
    \centering
    \includegraphics[width=0.9\linewidth]{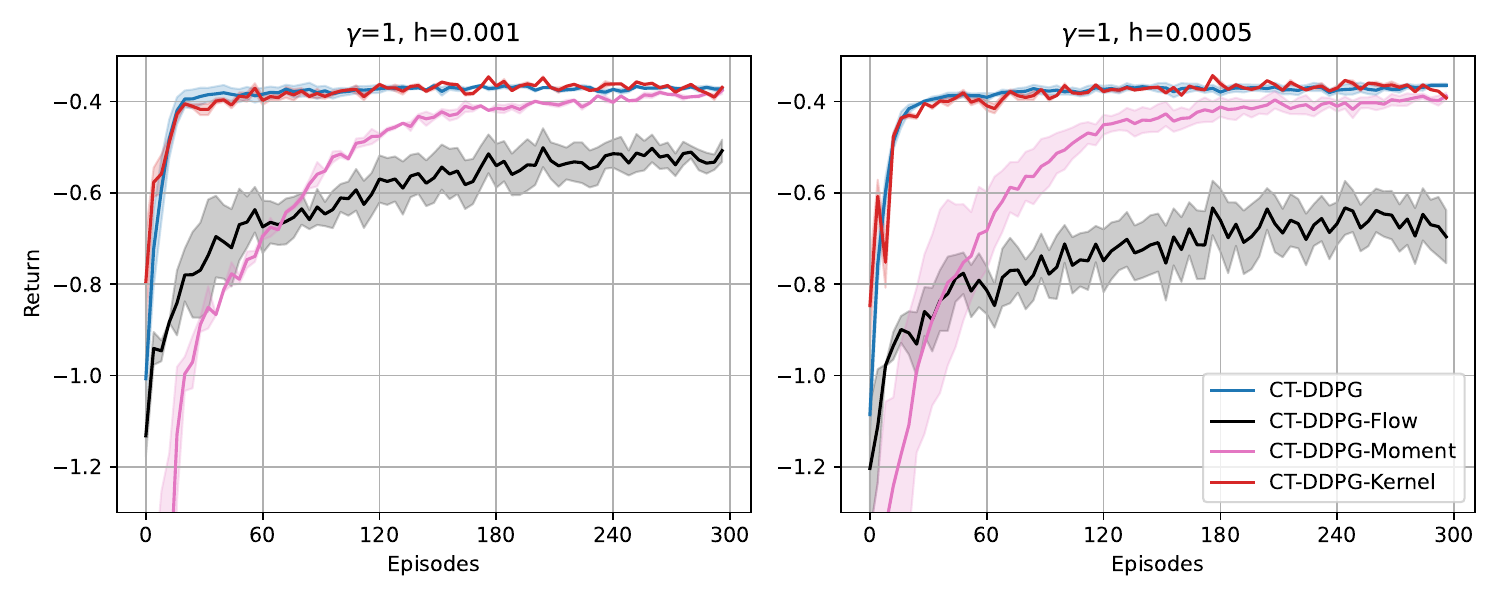}
    \caption{Results for C-S model with $\gamma=1$. Different designs of mean-field distribution embeddings in CT-DDPG lead to different convergence rates.}
    \label{fig:cs_non_lq}
\end{figure}

\begin{figure}[!ht]
    \centering
    \includegraphics[width=0.9\linewidth]{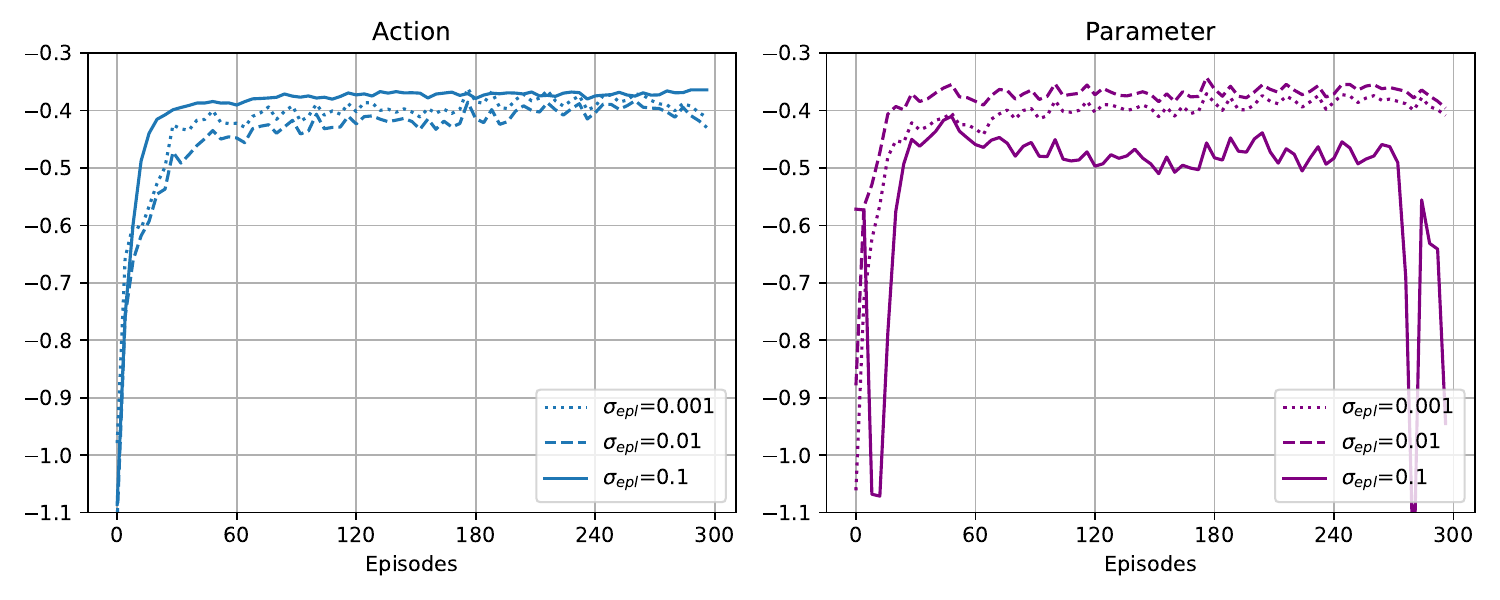}
    \caption{Results for C-S model with $\gamma=1, h=0.0005$, where CT-DDPG is implemented with different exploration strategies.}
    \label{fig:cs_non_lq_param}
\end{figure}

\paragraph{Results for LQ case.}

We first set $\gamma=0,C=1,c=0.1,\sigma=0.3,T=1$ in C-S dynamics and the initial state distribution is $\mu_0=\text{Unif}([0,1]^n)\times\text{Unif}([-1,1]^n)$.
The optimal policy is given explicitly in \eqref{eq:CS_LQ}. 
We shall examine the effect of different discretization step sizes; the corresponding results are reported in \Cref{fig:cs_lq}.

Observe that CT-DDPG achieves the fastest convergence to the optimal value, despite not explicitly exploiting the LQ structure. In contrast, AC and q-Learning adopt an LQ-based parameterization and nevertheless underperform CT-DDPG, further demonstrating the effectiveness of the proposed method. 


\paragraph{Results for non-LQ case.}

We further set $\gamma=1$, which imposes a non-trivial coupling between $z$ and $v$. The results of different distribution embedding methods in CT-DDPG are shown in \Cref{fig:cs_non_lq}.
We observe that both CT-DDPG and CT-DDPG-Kernel converge rapidly, indicating that a learnable neural network feature map can achieve performance comparable to that obtained when prior information is incorporated.
This demonstrates the efficiency and robustness of the proposed approach, as well as its compatibility with the deep RL framework. 
Although CT-DDPG-Moment uses only generic moment embeddings of the distribution and therefore converges more slowly, it still eventually attains the same optimal return as CT-DDPG. This further illustrates the effectiveness of deep-RL-based methods for MFC problems.
For CT-DDPG-Flow, it converges rather slowly, and the performance gap becomes more pronounced as $h$ decreases.
As discussed in \Cref{sec:alg}, although the value function can be directly characterized via Theorem \ref{thm:dpg_flow}, the martingale-based temporal difference loss induced by Theorem \ref{thm:martingale} provides a more informative signal for learning the decoupled value function, thereby substantially accelerating convergence.

To compare different exploration strategies, we evaluate CT-DDPG with both action-space and parameter-space exploration under varying exploration noise scales. As shown in \Cref{fig:cs_non_lq_param}, action-space exploration is more robust to the choice of noise scale, both in terms of convergence rate and final performance. In contrast, parameter-space exploration may suffer from severe instability when the perturbation scale is not properly tuned. Overall, with appropriate tuning, the two approaches achieve comparable performance in standard MFC problems.

\subsection{Optimal liquidation with  trade crowding}\label{subsec:liquidation}

In this section, we examine the performance of Algorithm \ref{alg:ddpg} in a portfolio liquidation problem with trading crowd, as studied in \cite{acciaio2019extended, reisinger2024fast, picarelli2025extended}. In this problem, a large number of market participants aim to liquidate their positions in the same asset by a fixed terminal time $T$, while accounting for the permanent price impact induced by their trading actions. The resulting cooperative equilibrium gives rise to an extended MFC problem for a representative agent.


Let $(\alpha_t)_{t\in [0,T]}$ denote the trading speed chosen by the representative agent. The state dynamics of the
extended MFC problem are given by: for all $t\in [0,T]$,  
\begin{equation}\label{eq:liquidation}
    \d Q_t = \alpha_t \d t,
    \quad
     \d S_t = \lambda \E[\alpha_t]\d t + \sigma\d W_t.
\end{equation}
Here 
$Q=(Q_t)_{t\in [0,T]}$ denotes the inventory process with a random initial state $Q_0$  representing the initial
inventories for all participants,
  $S=(S_t)_{t\in [0,T]}$
is the asset price process,   
the term $\lambda \sE[\alpha_t]$
with $\lambda \ge 0 $ 
captures the permanent market impact on the asset price induced by the aggregate trading of all participants,  and  $W$ is a one-dimensional Brownian motion   representing exogenous market noise.
The   objective of the agent is    to maximize the following reward functional 
over all adapted trading speeds: 
\begin{equation}
\label{eq:liquidation_reward}
    J (\alpha)= \E\left[\int_0^T -(Q_t^2+\alpha_tS_t+c \alpha_t^2)\d t + Q_T(S_T-CQ_T)\right],
\end{equation}
where the term 
$\alpha_t(S_t+c\alpha_t)$ represents the instantaneous liquidation cost with $c>0$ representing the linear temporary market impact, 
$Q^2_t$  penalizes inventory risk over time, and 
$Q_T(S_T-CQ_T)$  with $C>0$ is the liquidation value of the remaining inventory at terminal time.

Note that \eqref{eq:liquidation} and  \eqref{eq:liquidation_reward} constitute an LQ extended MFC  problem and the  optimal policy is given by (see e.g., \citep{yong2017linear}):
\begin{equation}
    \varphi^*(t,s,q,\mu)=-\frac{1}{c}\left(P_C(t)(q-\E_{q'\sim\mu}[q'])+P_{C-\frac{\lambda}{2}}(t)\E_{q'\sim\mu}[q']\right),
\end{equation}
where  for all $r\in\R$,
$(P_r(t))_{t\in [0,T]}$ satisfies 
$\frac{\d}{\d t}P_r(t)-\frac{1}{c}{P_r(t)}^2+1=0$ with $P_r(T)=r$.
The model architectures and training hyperparameters are the same as described in the previous section.
We set $\lambda=0.2, \sigma=0.3,c=0.1, C=1,T=1, h=0.0005$ and the initial mean-field distribution is $\mu_0=\text{Unif}([0.8,1.2])\times\text{Unif}([0.8,1.2])$.

The results of CT-DDPG with different exploration mechanisms are shown in \Cref{fig:liquidation}. 
Action-space exploration converges rapidly and robustly to the optimal value under various exploration noise scales, demonstrating its effectiveness for the extended MFC problems.
Parameter-space exploration is still    sensitive to noise scale, but it exhibits a faster convergence if the noise scale is properly tuned. Due to the specific structure of the liquidation problem, action space exploration does not change the expected control, and therefore provides limited exploration for the price process in \eqref{eq:liquidation}, leading to slightly slower convergence. This illustrates the potential advantage of parameter-space exploration in certain extended MFC problems, provided that the hyperparameters are chosen appropriately.

\begin{figure}[!ht]
    \centering
    \includegraphics[width=0.9\linewidth]{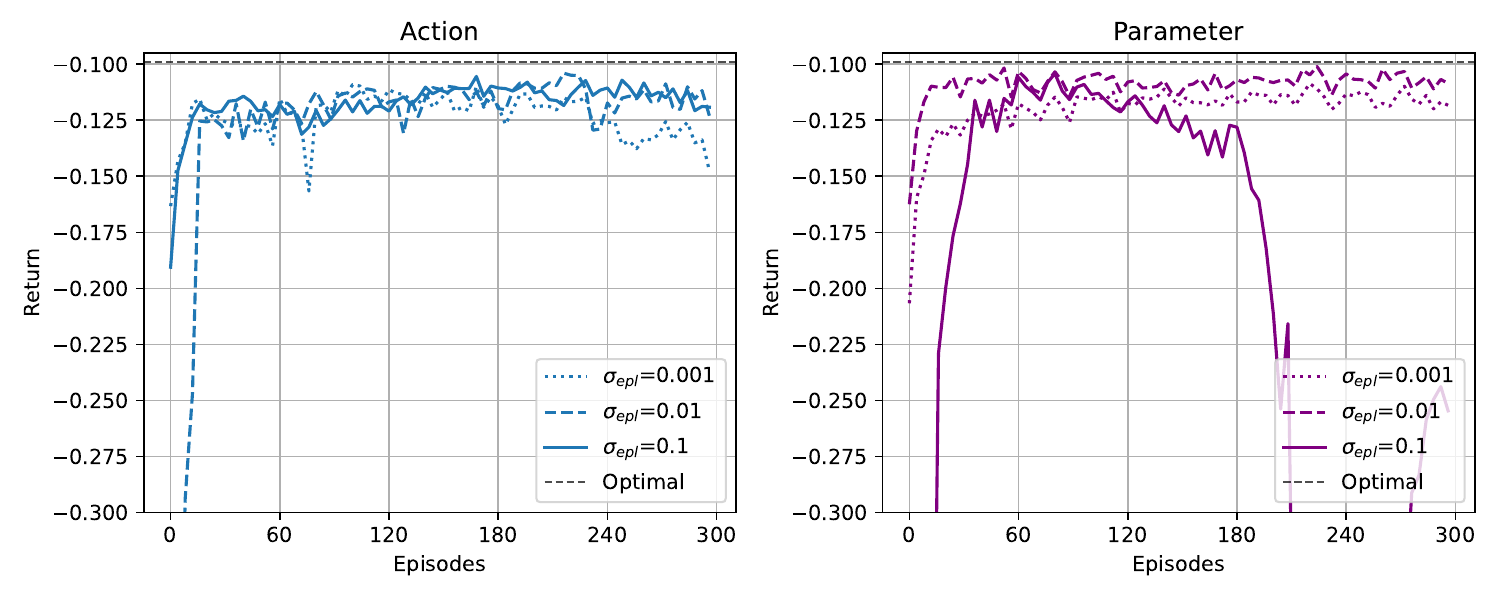}
    \caption{Results of CT-DDPG for liquidation problem with $h=0.0005$.}
    \label{fig:liquidation}
\end{figure}

In summary, CT-DDPG exhibits superior performance in terms of convergence speed and stability across all environment settings, verifying the efficiency and robustness of our method.

\section{Proofs}

\subsection{Proofs of Technical Results}

This section establishes several technical results used in the proofs of   Theorems \ref{thm:sensitivity} and \ref{thm:sensitivity_critic}.

We first recall that when  
  the function $V(\cdot, \cdot,\theta)$ 
  defined in \eqref{eq:cost_theta}
  is in $C^{1,2 }([0, T]   
 \times \mathcal{P}_2(\mathbb{R}^n))$,
   it is the unique solution to  an associated  linear  PDE.

\begin{lemma}
\label{lemma:pde}
    Suppose Assumption  \ref{assum:continuity} holds.
    Let 
     $\theta\in \sR^k$ 
   and assume  
     $
    V(\cdot, \cdot,\theta)   \in C^{1,2 }([0, T]  \times   \mathcal{P}_2(\mathbb{R}^n))$.
Then $V(\cdot,\cdot, \theta)$ is the unique classical solution to the following linear PDE:  
\begin{align} 
\label{eq:pde_v}
\begin{split} 
 & \cL^\theta[ w](t, \mu) +\bar r(t,  \mu,\theta) =0,
 \quad 
 \forall (t,  \mu)\in [0,T)  \times \cP_2(\sR^n),
\\
& w(T,  \mu) = \bar  g(  \mu),
 \quad \forall   \mu \in  \cP_2(\sR^n).
  \end{split}
\end{align}
 
\end{lemma}

Using Lemma  \ref{lemma:pde},
the following proposition 
characterizes   the performance difference between two 
 value functions $V(\cdot, \cdot,\theta)$ and $  V(\cdot,\cdot,{\theta'})$.
It extends the performance-difference lemmas in
\cite[Lemma 6.1]{cheng2025deterministic} and \cite[Lemma 3.2]{sethi2025entropy}
from classical control problems to the more general setting of McKean--Vlasov dynamics.

\begin{prop}
\label{prop:performance_difference}
 Suppose Assumption   \ref{assum:continuity} holds. 
  Let 
     $\theta\in \sR^k$ 
     and assume  
     $
    V^\theta\coloneqq V(\cdot,\cdot,\theta)   \in C^{1,2}([0, T]   
 \times \mathcal{P}_2(\mathbb{R}^n))$.
  For all $(t,  \mu)\in [0,T]\times   \cP_2(\sR^n)$
 and $\theta'\in \sR^k$,
 \begin{align}
 \begin{split}
 &  V^{\theta'}(t,  \mu )
  -V^{\theta }(t,  \mu )
  \\
  &\quad 
  =  
     \int_t^T
   e^{-\beta(s-t)} 
   \left((  \mathcal L^{  \theta'}[V^\theta] 
   -\mathcal L^{\theta }[V^\theta])
    (s,     \sP^{t,\mu,\theta'}_s ) 
    +\bar r\left(s, \sP^{t,\mu,\theta'}_s,\theta'\right) 
   - \bar r\left(s, \sP^{t,\mu,\theta'}_s,  \theta\right)
   \right)
   \d s. 
 \end{split}
 \end{align}

\end{prop}

\begin{proof}
    Fix  $\theta,\theta'\in \mathbb {R}^k$, $(t, \mu) \in [0,T]\times   \cP_2(\sR^n)$, and 
     $\xi\in L^2(\cF_t, \sR^n)$ with  $ \xi \sim\mu$.  
    Let $  X^{t,\xi, \theta'} $
        be the  solution to \eqref{eq:sde_mv_population} with the parameter $\theta'$ and recall  
        $\sP^{t,\mu, \theta'}_s=\sP_{X^{t,\xi, \theta'}_s}$ for all $s\in [t,T]$. 
        By the definition \eqref{eq:cost_theta} of $V^{\theta'} (t,  \mu )$, 
    \begin{align}
    \label{eq:performance_difference1}
    \begin{split}
       & V^{\theta'} (t, \mu )-V^{\theta}(t, \mu) 
       \\
        & =
\int_t^T e^{-\beta(s-t)} \bar r\left(s, \sP^{t,\mu,\theta'}_s,\theta'\right) \d s + 
   e^{-\beta(T-t)}  \bar g ( \sP^{t,\mu,\theta'}_T  ) 
   -V^{\theta}(t, \mu) 
\\  
&= 
   e^{-\beta(T-t)}  V^{\theta}\left(T,  \sP^{t,\mu,\theta'}_T\right) 
-V^{\theta}(t,\sP^{t,\mu,\theta'}_t) 
 + \int_t^T e^{-\beta(s-t)} \bar r\left(s, \sP^{t,\mu,\theta'}_s,\theta'\right) \d s,
\end{split}
    \end{align}
  where    the last identity used the fact that $V^\theta(T, \mu)=\bar g( \mu) $ and 
  $\sP^{t,\mu,\theta'}_t =\mu$.
As    $V^\theta \in C^{1,2 }([0,T]\times   \cP_2(\sR^n))$,
  applying   It\^{o}'s formula
(e.g.~\cite[Proposition 5.102]{carmona2018probabilistic}) to 
$s\mapsto e^{-\beta (s-t)}  V^\theta(s, \sP^{t,\mu,\theta'}_s )$
yields 
\begin{align*}
   &  e^{-\beta(T-t)}  V^{\theta}\left(T,  \sP^{t,\mu,\theta'}_T\right) 
-V^{\theta}(t,\sP^{t,\mu,\theta'}_t)
   = 
     \int_t^T
   e^{-\beta(s-t)} \mathcal L^{  \theta'}[V^\theta] (s, \sP^{t,\mu,\theta'}_s)\d s  
 \\
 & \quad
 = 
  \int_t^T
   e^{-\beta(s-t)} 
   (
   \mathcal L^{  \theta'}[V^\theta] 
   -\mathcal L^{\theta }[V^\theta]
   +\mathcal L^{\theta }[V^\theta]
   )
   (s,     \sP^{t,\mu,\theta'}_s )\d s  
   \\
   &\quad 
    = 
      \int_t^T
   e^{-\beta(s-t)} 
   \left((  \mathcal L^{  \theta'}[V^\theta] 
   -\mathcal L^{\theta }[V^\theta])
    (s,     \sP^{t,\mu,\theta'}_s ) 
   - \bar r\left(s, \sP^{t,\mu,\theta'}_s,   \theta\right)
   \right)
   \d s, 
\end{align*}
where   $\mathcal L^{  \theta'} [V^\theta]  $ 
and   $\mathcal L^{\theta } [V^\theta]  $ are 
defined in   \eqref{eq:generator}, 
and the last identity used the PDE \eqref{eq:pde_v}. 
This along with 
\eqref{eq:performance_difference1}
proves the desired result. 
\end{proof}

We further prove the joint continuity of the state law with respect to time and the parameter.

\begin{lemma}
\label{lemma:flow_continuity}
 Suppose Assumption    \ref{assum:continuity}   
 holds.  
 For all $(t,  \mu)\in [0,T] \times \cP_2(\sR^n)$,
 the map 
 $[t,T]\times \sR^k\ni (s,\theta)\mapsto \sP^{t,\mu,\theta}_s\in \cP_2(\sR^n)$ is continuous.  
   
\end{lemma}

\begin{proof}
Fix $(t,  \mu)\in [0,T] \times \cP_2(\sR^n)$. 
By Assumption    \ref{assum:continuity},
for all $\theta$ in a  compact set, 
the coefficients $b$ and $\sigma$ are Lipschitz continuous 
and of linear growth, 
and hence Assumptions (H1)-(H4) in \cite{ma2025continuous} hold (the Lyapunov condition (H2) holds with the function $V(x,\mu)=1+|x|^2+M_2(\mu)^2)$.
Thus by \cite[Theorem 3.1]{ma2025continuous},
$\theta\mapsto \sP^{t,\mu,\theta}_s$ is continuous for all $s\in [t,T]$.

We now claim that 
 for any bounded set  $ K\subset \sR^k$, 
there exists a constant $C_K\ge 0$ such that for all $s,r\in [t,T]$ and $\theta\in K$, 
\begin{equation}
\label{eq:holder_continuous_t}
W_2( \sP^{t,\mu,\theta}_r, \sP^{t,\mu,\theta}_s)\le C_K|r-s|^{1/2}.
\end{equation} 
This along with the continuity of $\theta\mapsto \sP^{t,\mu,\theta}_s$  implies the joint continuity of 
$(\theta,s)\mapsto \sP^{t,\mu,\theta}_s$.  

To show \eqref{eq:holder_continuous_t}, assume without loss of generality that $t\le s<r \le T$,
  let $ \xi\in  L^2(\mathcal{F}_t; \sR^n)$ with $\mu=\sP_{\xi}$,
and let $C\ge 0$ be a generic constant independent of $s,r$ and $\theta$.
 Then by \eqref{eq:sde_mv_population},
\begin{align*}
&W_2^2( \sP^{t,\mu,\theta}_r, \sP^{t,\mu,\theta}_s)
\le \sE[|X^{t,\xi, \theta}_r-X^{t,\xi, \theta}_s|^2]
\\
&
=  \sE 
\left[\left| \int_s^r b(u,X^{t, \xi, \theta}_u, \sP_{X^{t, \xi,\theta}_u}, \theta)\d u+\int_s^r  \sigma(u,X^{t, \xi,\theta}_u, \sP_{X^{t, \xi, \theta}_u}, \theta)\d W_u\right|^2\right] 
\\
&
\le
C\left( \sE 
\left[  \int_s^r |b(u,X^{t, \xi, \theta}_u, \sP_{X^{t, \xi,\theta}_u}, \theta)|^2\d u \right]+
\sE 
\left[  \int_s^r  |\sigma(u,X^{t, \xi,\theta}_u, \sP_{X^{t, \xi, \theta}_u}, \theta)|^2\d u  \right]
\right)
\\
&\le 
C(r-s)\left(   1 +\sup_{u\in [s,r]}\sE[|X^{t, \xi, \theta}_u|^2] 
\right)\le C(r-s),
\end{align*}
where we have used the linear growth of $b$ and $\sigma$ in   \Cref{assum:continuity} and 
the fact that $\sE[|X^{t, \xi, \theta}_s|^2]\le C$, uniformly with respect to $s\in [t,T]$ and $\theta \in K$. This proves the  desired estimate  \eqref{eq:holder_continuous_t}.
\end{proof}

\subsection{Proofs of 
Theorems 
\ref{thm:sensitivity}, 
\ref{thm:sensitivity_critic_ncessary}, \ref{thm:sensitivity_critic}
and \ref{thm:martingale}}
 
We now prove Theorem \ref{thm:sensitivity}
based on Proposition \ref{prop:performance_difference} and Lemma \ref{lemma:flow_continuity}.

\begin{proof}[Proof of Theorem \ref{thm:sensitivity}]
To simplify the notation, we write 
     $
    V^\theta\coloneqq V(\cdot,\cdot,\theta) $ for all  
     $\theta\in \sR^k$. 
Since $\partial_\theta V^{\theta }(t,  \mu ) = (\partial_{\theta_1}V^{\theta }(t,  \mu ), \ldots, \partial_{\theta_k}V^{\theta }(t,  \mu ))^\top$, it suffices to prove that for any fixed   $\theta,\theta'\in \mathbb {R}^k$ and  $(t, \mu) \in [0,T]\times   \cP_2(\sR^n)$, 
\begin{align}
 \begin{split}
&  \frac{\d }{\d \epsilon }V^{\theta+\epsilon \theta'}(t, \mu )
  \bigg\vert_{\epsilon=0}
   =  
     \int_t^T
   e^{-\beta(s-t)} 
   \partial_\theta 
   A [V^\theta] 
    (s, \sP^{t,\mu,\theta}_s ,\theta) \, 
   \d s\cdot \theta'.
 \end{split}
 \end{align}

To this end,
let  $\xi\in L^2(\cF_t, \sR^n)$,
    and  for all $\epsilon \in [-1,1]$,
    let $  (X^{t,\xi, \theta+\epsilon \theta'}_s)_{s\in [t,T]} $
         be the solution to \eqref{eq:sde_mv_population} with the parameter $\theta+\epsilon \theta'$,
         and 
let $\sP^{t,\mu,\theta+\epsilon \theta'}_s = \sP_{X^{t,\xi, \theta+\epsilon \theta'}_s}$.  
For all $\epsilon\in[-1,1]$,
by Proposition \ref{prop:performance_difference}
and the fundamental theorem of calculus, 
   \begin{align}
   \label{eq:derivative_proof1}
 \begin{split}
& \frac{1}{ \epsilon} \left(V^{\theta+\epsilon \theta'}(t,  \mu )
  -V^{\theta }(t,  \mu )\right) 
  \\
  &=
    \frac{1}{ \epsilon}  \int_t^T
   e^{-\beta(s-t)} 
   \left(    
   A[V^\theta] (s, \sP^{t,\mu,\theta+\epsilon \theta'}_s, \theta+\epsilon \theta')
   -
   A[V^\theta] (s, \sP^{t,\mu,\theta+\epsilon \theta'}_s, \theta)
   \right)
   \d s
  \\
& =  
     \int_t^T
   e^{-\beta(s-t)} 
   \left(  
   \int_0^1  
   \partial_\theta 
   A [V^\theta] 
    (s, \sP^{t,\mu,\theta+\epsilon \theta'}_s ,\theta+u \epsilon\theta')   \d u  
   \right)
   \d s  \cdot  \theta'.
 \end{split}
 \end{align}
Observe that for all $s\in [t,T]$
and $u\in [0,1]$,
by Lemma \ref{lemma:flow_continuity},
$\lim_{\epsilon\to 0}\sP^{t,\mu,\theta+\epsilon \theta'}_{s}=\sP^{t,\mu,\theta}_{s}$ for all $s\in [0,T]$,
and hence by Assumption 
\ref{assum:smoothness},
$$
 \lim_{\epsilon\to 0}\partial_\theta 
   A [V^\theta] 
    (s, \sP^{t,\mu,\theta+\epsilon \theta'}_s ,\theta+u \epsilon\theta') 
    =\partial_\theta 
   A [V^\theta] 
    (s, \sP^{t,\mu,\theta}_s ,\theta ). 
$$
Moreover, 
  the continuity of $(s,\theta)\mapsto \sP^{t,\mu,\theta}_{s}$ and the compactness of $[t,T]\times [-1,1]$ imply that 
$\{\sP^{t,\mu,\theta+\epsilon \theta'}_s\mid (s,\epsilon)\in [t,T]\times [-1,1]\}$ is compact in $\cP_2(\sR^n)$.  
Hence using the dominated convergence theorem and passing  $\epsilon \to 0$ 
in \eqref{eq:derivative_proof1} yields the desired identity. 
\end{proof}

We further prove Theorems  \ref{thm:sensitivity_critic_ncessary} and  \ref{thm:sensitivity_critic}.

\begin{proof}[Proof of Theorem \ref{thm:sensitivity_critic_ncessary}]
$V^\theta$ and $q^\theta$ satisfy \eqref{eq:bellman_param_necessary}  
due to Lemma \ref{lemma:pde} and the definition of $q^\theta= A[V^\theta]$  in \eqref{eq:advantage_rate}.
The functions 
$V^\theta$ and $q^\theta$  satisfy 
\eqref{eq:constant_flow_necessary} 
due to 
It\^o's formula (e.g.~\cite[Proposition 5.102]{carmona2018probabilistic}).
\end{proof}

\begin{proof}[Proof of Theorem \ref{thm:sensitivity_critic}]
    
    For all 
    $(t,\mu,\theta')\in [0,T]\times \cP_2(\sR^n)\times \cO_{t,\mu}(\theta)$
    and 
    $s\in [t,T]$,
    by It\^o's formula (e.g.~\cite[Proposition 5.102]{carmona2018probabilistic}),
    \begin{equation*}
        \begin{aligned}
            &e^{-\beta s}\hat{V}(s,\P^{t,\mu,\theta'}_s) - e^{-\beta t}\hat{V}(t,\P^{t,\mu,\theta'}_t)+\int_t^se^{-\beta u}\hat{r}(u,\P^{t,\mu,\theta'}_u,\theta')\,\d u \\
            &=\int_t^se^{-\beta u}\bigg\{\partial_t\hat{V}(u,\P^{t,\mu,\theta'}_u)-\beta\hat{V}(u,\P^{t,\mu,\theta'}_u)  +\sE_{\xi\sim \P^{t,\mu,\theta'}_u }\bigg[b(u,\xi,\P^{t,\mu,\theta'}_u,\theta')\cdot \partial_\mu \hat{V}(u,\P^{t,\mu,\theta'}_u)(\xi)\\
            &\qquad\qquad   +\frac{1}{2}(\sigma\sigma^\top) (u,\xi,\P^{t,\mu,\theta'}_u,\theta') :
            \partial_v\partial_\mu \hat{V}(u,\P^{t,\mu,\theta'}_u) (\xi)\bigg]+\bar{r}(u,\P^{t,\mu,\theta'}_u,\theta')\bigg\}\,\d u \\
            &= \int_t^se^{-\beta u}
             A [\hat{V}](u,\P^{t,\mu,\theta'}_u,\theta')
             \,\d u.
        \end{aligned}
    \end{equation*}
This along with \eqref{eq:constant_flow} implies 
    \begin{equation}\label{eq:q_condition}
        \int_t^s
        e^{-\beta u}(A [\hat{V}](u,\P^{t,\mu,\theta'}_u,\theta')-\hat{q} (u,\P^{t,\mu,\theta'}_u,\theta'))\,\d u=0,\quad  \forall s\in [t,T].
    \end{equation}
By Assumption 
 \ref{assum:smoothness}
 and Lemma   \ref{lemma:flow_continuity},
    $u\mapsto e^{-\beta u}(A [\hat{V}](u,\P^{t,\mu,\theta'}_u,\theta')-\hat{q} (u,\P^{t,\mu,\theta'}_u,\theta'))$
is continuous, 
and hence by \eqref{eq:q_condition},
$A [\hat{V}](u,\P^{t,\mu,\theta'}_u,\theta')=\hat{q} (u,\P^{t,\mu,\theta'}_u,\theta')$
for all $u\in [t,T]$.
Setting $u=t$ implies 
$$
A [\hat{V}](t,\mu,\theta')=\hat{q} (t,\mu,\theta'),
\quad \forall  (t,\mu,\theta')\in [0,T]\times \cP_2(\sR^n)\times  \cO_{t,\mu}(\theta).
$$
Taking $\theta'=\theta$ and using  \eqref{eq:bellman_param}  shows that $\hat V$ is a classical solution to  
the following linear PDE: 
for all $(t,\mu)\in [0,T)\times \cP_2(\sR^n)$,
$
A [\hat{V}](t,\mu,\theta)=
\mathcal L^\theta [\hat{V}](t,\mu)
+\bar r(t,\mu,\theta)=0,
$
and $\hat V(T,\mu)=\bar g(\mu)$.
By Lemma \ref{lemma:pde},
$\hat V(t,\mu)=V(t,\mu, \theta)$
for all $(t,\mu)\in [0,T]\times \cP_2(\sR^n)$. 
This further implies that 
$\hat q(t,\mu,\theta')=A[V(\cdot,\cdot,\theta)](t,\mu,\theta')$ for all $\theta'\in \cO_{t,\mu}(\theta)$.
\end{proof}

\begin{proof}[Proof of Theorem \ref{thm:martingale}]
 It suffices to prove that \eqref{eq:martingale} implies that $\hat V$ given in  \eqref{eq:decompose_V} 
 satisfies \eqref{eq:constant_flow_policy} in 
 Theorem \ref{thm:dpg_flow}.
 To see it, 
 by \eqref{eq:martingale},
  \begin{align*}
    \begin{split}
     &e^{-\beta t}\hat{V}_D(t,X_t^{t,\xi,\theta'},\sP_t^{t,\mu,\theta'}) =\sE[e^{-\beta s}\hat{V}_D(s,X_s^{t,\xi,\theta'},\sP_s^{t,\mu,\theta'})\mid \cF_t]
        \\
        &\quad +\sE\left[\int_t^se^{-\beta u}(r-\hat{q}_D)(u,X_u^{t,\xi,\theta'},\varphi_{\theta'}(u,X_u^{t,\xi,\theta'},\sP_u^{t,\mu,\theta'}),
        (\id, \varphi_{\theta'} (u,\cdot, \sP_u^{t,\mu,\theta'}))_\sharp \sP_u^{t,\mu,\theta'})\d u\,\bigg\vert \, \cF_t\right], 
    \end{split}
\end{align*}  
from which by taking expectations on both sides  yields
 \begin{align*}
    \begin{split}
     &e^{-\beta t}\hat{V}(t,\mu) =\sE[e^{-\beta s}\hat{V}_D(s,X_s^{t,\xi,\theta'},\sP_s^{t,\mu,\theta'})]
        \\
        &\quad +\sE\left[\int_t^se^{-\beta u}(r-\hat{q}_D)(u,X_u^{t,\xi,\theta'},\varphi_{\theta'}(u,X_u^{t,\xi,\theta'},\sP_u^{t,\mu,\theta'}),
        (\id, \varphi_{\theta'} (u,\cdot, \sP_u^{t,\mu,\theta'}))_\sharp \sP_u^{t,\mu,\theta'})\d u \right].
    \end{split}
\end{align*}  
This along with  Fubini's theorem, $\sP_{X_s^{t,\xi,\theta'}}=\sP_s^{t,\mu,\theta'}$ 
 and the definitions of $\hat V$, $\hat q$ and $r^\varphi$ yields 
 \eqref{eq:constant_flow_policy}. 
 The desired conclusion then follows from 
  Theorem \ref{thm:dpg_flow} and 
the chain rule \eqref{eq:chain_rule_q}. 
   \end{proof}

\noindent \textbf{Acknowledgements and Disclosure of Funding}  

\vspace{1mm}

\noindent Huy\^en  Pham is supported by the Soci\'et\'e G\'en\'erale Chair ``Risques Financiers", FiME (Laboratory of Finance and Energy Markets), and the EDF–CACIB Chair ``Finance and Sustainable Development''. 
Yufei Zhang’s research is supported by CNRS-Imperial “Abraham de Moivre” International Research Lab in Mathematics.

\bibliographystyle{abbrv}
\bibliography{ref}

\end{document}

%% file: math_commands.tex
\usepackage{amsmath,amsfonts,bm}

\def\eqref#1{equation~\ref{#1}}

\def\1{\bm{1}}

\DeclareMathAlphabet{\mathsfit}{\encodingdefault}{\sfdefault}{m}{sl}
\SetMathAlphabet{\mathsfit}{bold}{\encodingdefault}{\sfdefault}{bx}{n}

\def\sF{{\mathbb{F}}}

\def\sP{{\mathbb{P}}}

\def\sR{{\mathbb{R}}}

\newcommand{\E}{\mathbb{E}}

\newcommand{\R}{\mathbb{R}}

\newcommand{\Var}{\mathrm{Var}}



%% file: def.tex
\newtheorem{thm}{Theorem}[section]

\newtheorem{lemma}[thm]{Lemma}

\newtheorem{prop}[thm]{Proposition}
\theoremstyle{definition}

\crefname{equation}{eq.}{eqs.}
\Crefname{equation}{Eq.}{Eqs.}
\crefname{Assumption}{asp.}{asps.}
\Crefname{Assumption}{Asp.}{Asps.}
\crefname{thm}{thm.}{thms.}
\Crefname{thm}{Theorem}{Thms.}
\crefname{prop}{prop.}{props.}
\Crefname{prop}{Prop.}{Props.}
\crefname{definition}{def.}{defs.}
\Crefname{definition}{Def.}{Defs.}
\crefname{algorithm}{alg.}{algs.}
\Crefname{algorithm}{Alg.}{Algs.}
\crefname{figure}{fig.}{figs.}
\Crefname{figure}{Fig.}{Figs.}
\crefname{section}{sec.}{secs.}
\Crefname{section}{Sec.}{Secs.}